\documentclass[twoside,11pt]{article}

\usepackage{xcolor}
\usepackage{pgf,tikz}
\usepackage{mathrsfs}
\usetikzlibrary{arrows}
\usepackage{amssymb,amsmath}
\usepackage{amsthm}
\usepackage{hyperref}
\usepackage{mathrsfs}
\usepackage{textcomp}
\usepackage{verbatim}
\usepackage{esint}
\usepackage{sectsty}
\usepackage{cases}
\usepackage{fancyhdr}
\hypersetup{
    colorlinks,%
    citecolor=black,%
    filecolor=black,%
    linkcolor=black,%
    urlcolor=black
}

\makeatletter

\newdimen\bibspace
\renewenvironment{thebibliography}[1]{%
 \section*{\refname 
       \@mkboth{\MakeUppercase\refname}{\MakeUppercase\refname}}%
     \list{\@biblabel{\@arabic\c@enumiv}}%
          {\settowidth\labelwidth{\@biblabel{#1}}%
           \leftmargin\labelwidth
           \advance\leftmargin\labelsep
           \itemsep\bibspace
           \parsep\z@skip     %
           \@openbib@code
           \usecounter{enumiv}%
           \let\p@enumiv\@empty
           \renewcommand\theenumiv{\@arabic\c@enumiv}}%
     \sloppy\clubpenalty4000\widowpenalty4000%
     \sfcode`\.\@m}
    {\def\@noitemerr
      {\@latex@warning{Empty `thebibliography' environment}}%
     \endlist}

\makeatother

\makeatletter

\newtheorem{thm}{Theorem}[section]
\newtheorem{lem}[thm]{Lemma}
\newtheorem{prop}[thm]{Proposition}

\newtheorem{cor}[thm]{Corollary}
\newtheorem{rem}[thm]{Remark}

\def\XXint#1#2#3{{\setbox0=\hbox{$#1{#2#3}{\int}$}
  \vcenter{\hbox{$#2#3$}}\kern-.5\wd0}}

\newcommand{\al}{\alpha}                \newcommand{\lda}{\lambda}
\newcommand{\om}{\Omega}                \newcommand{\pa}{\partial}
\newcommand{\va}{\varepsilon}           \newcommand{\ud}{\mathrm{d}}
\newcommand{\be}{\begin{equation}}      \newcommand{\ee}{\end{equation}}
\newcommand{\w}{\omega}                 
              
\newcommand{\R}{\mathbb{R}}

\begin{document}

\title{\textbf{Classification theorems for solutions of higher order boundary conformally invariant problems, I}}
\author{Liming Sun \  \  and \ \
Jingang Xiong\footnote{Supported in part by NSFC 11501034, NSFC 11571019,  Beijing MCESEDD (20131002701) and the Fundamental Research Funds for the central University  (2015NT03).}}

\date{\today}

\fancyhead{}
\fancyhead[CO]{\textsc{Classification Theorem}}
\fancyhead[CE]{\textsc{L. Sun \& J. Xiong}}

\fancyfoot{}

\fancyfoot[CO, CE]{\thepage}

\renewcommand{\headrule}{}

\maketitle

\begin{abstract} In this paper, we prove that nonnegative polyharmonic functions on the upper half space satisfying a  conformally invariant nonlinear boundary condition have to be the ``\emph{polynomials} plus \emph{bubbles}" form. The nonlinear problem is motivated by the recent studies of boundary GJMS operators and the $Q$-curvature in conformal geometry. The result implies that in the conformal class of the unit Euclidean ball there exist metrics with a single singular boundary point  which have flat $Q$-curvature and constant boundary $Q$-curvature. Moreover, all of such metrics are classified.  This phenomenon differs from that of boundary singular metrics which have flat scalar curvature and constant mean curvature, where the singular set contains at least two points.
A crucial ingredient of the proof is developing an approach to separate the higher order linear effect and the boundary nonlinear effect so that the kernels of the nonlinear problem are captured.
\end{abstract}

\tableofcontents

\section{Introduction}

In the classical paper \cite{CGS}, Caffarelli-Gidas-Spruck established the asymptotic behavior for local positive solutions of the elliptic equation $-\Delta u= n(n-2)u^{\frac{n+2}{n-2}}$, $n\ge 3$, near an isolated singularity. Consequently, they proved that any positive entire solution of the equation
has to be the form
\[
\left(\frac{\lda}{1+\lda^2|x-x_0|^2}\right)^{\frac{n-2}{2}} \quad \mbox{for some }\lda>0, ~ x_0\in \R^n.
\]
Particular interests of the above equation  lie in its relation to the Yamabe problem (see Lee-Park \cite{LP}).  Such Liouville type theorem has been extended to general conformally invariant nonlinear equations; see Lin \cite{Lin} and Wei-Xu \cite{WX} for higher order semi-linear equations, Chen-Li-Ou \cite{CLO}, Li \cite{Li04} and many others for integral equation, as well as Li-Li \cite{LL} for fully nonlinear second order elliptic equations.

Li-Zhu \cite{Li-Zhu} and Ou \cite{Ou} independently proved that any positive solution of
\begin{align}
\label{eq:2nd-a}
-\Delta u(x,t)&=0 \quad \mbox{in }\R^{n+1}_+:=\R^n\times (0,\infty),\\
-\pa_t u(x,0)&=(n-1)u^{\frac{n+1}{n-1}} \quad \mbox{on }\pa \R^{n+1}_+,
\label{eq:2nd-b}
\end{align}
where $n\ge 2$,  has to be the form
\be\label{eq:m1bubble}
\left(\frac{\lda}{\lda^2|x-x_0|^2+(\lda t+1)^2}\right)^\frac{n-1}{2}, \quad \lda > 0,x_0\in \R^n.
\ee  Throughout the paper $\pa \R^{n+1}_+$ does not contain the infinity.    See also Beckner \cite{Bec} and Escobar \cite{Es90} if $u$ is an extremal of the sharp Sobolev trace inequality, and  Li-Zhang \cite{Li-Zhang}, Jin-Li-Xiong \cite{JLX}  and references therein for related results.  The isolated singularity problem has been studied recently by Caffarelli-Jin-Sire-Xiong \cite{C+}, DelaTorre-Gonz\'alez \cite{DG} and DelaTorre-del Pino-Gonzalez-Wei \cite{DD+} as a special case. The nonlinear problem \eqref{eq:2nd-a}-\eqref{eq:2nd-b} arises from  a boundary Yamabe problem or Riemann mapping problem of Escobar \cite{Es}, sharp trace inequalities, nonlinear Neumann problems (see Cherrier \cite{Cher}), and etc.

By the work Feffermann-Graham \cite{FG}, Graham-Jenne-Mason-Sparling \cite{GJMS}, and  Graham-Zworski \cite{GZ}, there defines a class of conformally invariant operators on the conformal infinity of Poincar\'e-Einstein manifolds via scattering matrices. Such conformally invariant operators define fractional $Q$-curvatures.  By the work Caffarelli-Silvestre \cite{CaffS}, Chang-Gonz\'alez \cite{CG}, Yang \cite{Y} and Case-Chang \cite{CaC},  the boundary Yamabe problem mentioned above is the constant first order $Q$-curvature problem. If one considers the constant odd order $Q$-curvature problem on the conformal infinity of Poincar\'e ball or hyperbolic upper half space, it will lead to study positive solutions of nonlinear boundary value problem of polyharmonic equations
\begin{align}
\Delta^m u(x,t)=0 \quad& \mbox{in }\R^{n+1}_+,\label{eq:bi-laplace}\\
\pa_t \Delta^k u(x,0)=0, \quad (-1)^m\pa_t \Delta^{m-1}& u(x,0)= u^{\frac{n+(2m-1)}{n-(2m-1)}} \quad \mbox{on }\pa \R^{n+1}_+,\label{eq:bi-b1}
\end{align}
where $2\le 2m<n+1$ is an integer, $k=0,1,\dots, m-2$. One may view $(-1)^m\pa_t \Delta^{m-1}$ as $(-\pa_t)(-\Delta)^{m-1} \sim (-\Delta)^{\frac12}(-\Delta)^{m-1}$. Hence, the above problem connects to
\be \label{eq:core}
(-\Delta)^{\frac{2m-1}{2}} u=u^{\frac{n+(2m-1)}{n-(2m-1)}} \quad \mbox{in } \R^{n}.
\ee However, we will see that \eqref{eq:bi-laplace}-\eqref{eq:bi-b1} admits more solutions. Since we do not assume $u$ to be a minimizer or belong to some Sobolev space of $\R^{n+1}_+$, there is no information of $u$ near the infinity.

In this paper, we classify solutions of problem \eqref{eq:bi-laplace}-\eqref{eq:bi-b1} and the subcritical cases. Consider
\be \label{eq:thma-1}
\begin{cases}
\Delta^{m} u(x,t)=0 \quad &\text{in }\R^{n+1}_+,\\
\pa_t \Delta^k u(x,0)=0\quad &\text{on }\pa \R^{n+1}_+,\quad k=0,1,\dots, m-2, \\
(-1)^m\pa_t \Delta^{m-1} u(x,0)= u^{p} \quad &\text{on }\pa \R^{n+1}_+,
\end{cases}
\ee
where $2\le2m<n+1$ is an integer and  $1< p\leq \frac{n+(2m-1)}{n-(2m-1)}$. We will show the nonnegative solutions of this problem are the composition of the following "bubbles" and some polynomials
\be \label{eq:bubble}
U_{x_0,\lda}(X)=c(n,m) \int_{\R^n} \frac{t^{2m-1}}{(|x-y|^2+t^2)^{\frac{n+2m-1}{2}}}\left(\frac{\lda}{1+\lda^2|y-x_0|^2}\right)^{\frac{n-2m+1}{2}} \,\ud y
\ee
where $x_0\in \R^n$ and $\lda> 0$ and $c(n,m)>0$ is some normalizing constant. The presence of polynomial part is a new phenomenon.
More precisely
\begin{thm} \label{thm:main-a} Let $u\ge 0$ be a $C^{2m}(\R^{n+1}_+\cup \pa \R^{n+1}_+)$ solution of \eqref{eq:thma-1}.
 In case of that $m$ is even, we additionally suppose that $u(x,t)=o((|x|^2+t^2)^{\frac{2m-1}{2}})$ as $x^2+t^2\to\infty$. Then
\begin{itemize}
\item[(i)] If $p=\frac{n+(2m-1)}{n-(2m-1)}$, we have
\[
u(x,t)= U_{x_0,\lda}(x,t)+\sum_{k=1}^{m-1} t^{2k}P_{2k}(x),
\]
where $U_{x_0,\lda}$ is defined in \eqref{eq:bubble} for some $x_0\in \R^n$ and $\lda\ge 0$, and $P_{2k}(x)$ is a polynomial of degree $\leq 2m-2-2k$ satisfying $\liminf_{x\to \infty}P_{2k}(x)\ge 0$.
\item[(ii)] If $1< p< \frac{n+(2m-1)}{n-(2m-1)}$, we have
\[
u(x,t)=\sum_{k=1}^{m-1} t^{2k}P_{2k}(x),
\]
where $P_{2k}(x)\ge 0$ is a polynomial of degree $\leq 2m-2-2k$.
\end{itemize}

\end{thm}

\begin{rem}
 For $m=1$, $U_{x_0,\lda}$ defined \eqref{eq:bubble} equals \eqref{eq:m1bubble} up to a constant. For $m=2$, we have
\begin{align*}
U_{x_0,\lda}(X)
=C(n) \left(\frac{\lambda}{(1+\lambda t)^2+\lambda^2|x-x_0|^2}\right)^{\frac{n-3}{2}}+C(n)(n-3)t\left(\frac{\lambda}{(1+\lambda t)^2+\lambda^2|x-x_0|^2}\right)^{\frac{n-1}{2}}
\end{align*}
with $C(n)=[2(n-3)(n^2-1)]^{\frac{n-3}{6}}$.
\end{rem}
\begin{rem}\label{rem:extra}
 If $m$ is even and the growth condition is removed, there is another class of solutions
\begin{align}\label{eq:specialsolution}
H_a(x,t)=\frac{a}{(2m-1)!}t^{2m-1}+ a^{\frac{1}{p}}, \quad a\ge 0.
\end{align}
We conjecture that for even $m$, all solutions have to be
\[
\sum_{k=1}^{m-1} t^{2k}P_{2k}(x) + H_a(x,t) \quad \mbox{or} \quad \sum_{k=1}^{m-1} t^{2k}P_{2k}(x)+ U_{x_0,\lda}(x,t)
\]
if $p=\frac{n+2m-1}{n-2m+1}$, while only the former expression can happen if $1<p<\frac{n+2m-1}{n-2m+1}$.
\end{rem}

By conformally transforming the upper half space to the unit ball, Theorem \ref{thm:main-a} implies that  in the conformal class of the unit Euclidean ball there exist metrics with a single singular boundary point  which have flat $Q$-curvature and constant boundary $Q$-curvature. See Section \ref{sec:appl} of the paper for more details. When $m=1$, there is  no such metric which is singular on single boundary point because the polynomial part vanishes and the bubble is smooth at the infinity. Hence, boundary singular metrics have at least two singular points which is similar to the singular metrics on the unit sphere of constant scalar curvature; see Caffarelli-Gidas-Spruck \cite{CGS} and Schoen \cite{Schoen}. Other possible applications of Theorem \ref{thm:main-a} would be seen in Jin-Li-Xiong \cite{JLX3}, Li-Xiong \cite{LX} and references therein.

The proofs of Theorem \ref{thm:main-a} for $m=1$  by Li-Zhu \cite{Li-Zhu} or Ou \cite{Ou} rely on the maximum principle in order to use the moving spheres/planes method.  In contrast, for $m\ge 2$ the elliptic operators have nontrivial kernels and thus solutions of \eqref{eq:thma-1} could lose the maximum principle. To extract the kernels, we need to analyze the behavior of $u$ near the infinity. Due to the conformal invariance of equations, the $m$-Kelvin transform $u^*$ of $u$ with respect to the unit sphere  satisfies \eqref{eq:bi-laplace} and
\[
\pa_t \Delta^k u^*(x,0)=0, \quad (-1)^m\pa_t \Delta^{m-1} u^*(x,0)= |x|^{-\tau}{u^*}(x,0)^{p} \quad \mbox{in }\pa \R^{n+1}_+\setminus\{0\},
\]
where $k=0,1,\dots, m-2$, and  $\tau=[n+(2m-1)]-p[n-(2m-1)]\geq 0$. As Caffarelli-Gidas-Spruck \cite{CGS}, Lin \cite{Lin} and Wei-Xu \cite{WX} did, one may wish to show $|x|^{-\tau}{u^*}(x,0)^{p}\in L^1$ near $0$. However, since the linear equation itself would generate higher order singularities than the nonlinear term does, the methods of \cite{CGS}, \cite{Lin} and \cite{WX} seem not to be applicable to $m\ge 2$. Even worse, this is wrong when $m$ is even;
see for instance the $m$-Kelvin transform of $H_a$, $a>0$, in Remark \ref{rem:extra}. In fact, the method of \cite{CGS} is by constructing test function which is of second order equation nature. And it is unclear how to adapt the ODE analysis procedure of \cite{Lin} and \cite{WX} to our setting without information about the possible kernels. As the initial step, we prove that $u^*(x,0)$ belongs to $L^1$ (see Lemma \ref{lemma:babyversion}), and then by a Poisson extension we are able to capture the singularity generated by the linear equation. A Liouville type theorem (see Proposition \ref{prop:Ch-sing} and Theorem \ref{thm:positive_homo}) for polyharmonic functions with a homogeneous boundary data plays an important role. Our method of proof of Theorem \ref{thm:positive_homo} is very  flexible and can be easily adapted to polyharmonic functions with other  homogeneous boundary data. Next, by subtracting the linear effect we prove $|x|^{-\tau}{u^*}(x,0)^{p}\in L^1$, where the growth condition is assumed if $m$ is even. In this step a new method is developed.  In particular, if $p$ is less than the Serrin's exponent $\frac{n}{n-2m+1}$ we have to spend extra efforts. By a Neumann extension of $|x|^{-\tau}{u^*}(x,0)^{p}$ and making use of a boundary B\^ocher theorem (see Corollary \ref{cor:Bocher}), we prove
a crucial splitting result for $u$; see Proposition \ref{prop:splitting}. It captures the polynomials $\sum_{k=1}^{m-1} t^{2k}P_{2k}(x)$ and implies the maximum principle for $v(x,t):=u(x,t)-\sum_{k=1}^{m-1} t^{2k}P_{2k}(x)$ which is completely controlled by the nonlinear effect. Since $v(x,0)=u(x,0)$, $v$ satisfies a nonlinear integral equation. By Chen-Li-Ou \cite{CLO}, Li \cite{Li04}, or Dou-Zhu \cite{DZ}, $v(x,t)$ is then classified.

Our method of proof of Theorem \ref{thm:main-a} can be applied to constant fractional $Q$-curvature equation on the conformal infinity of hyperbolic upper half space, and can be applied to multiple nonlinear boundary conditions; see Chang-Qing \cite{CQ}, Branson-Gover \cite{BG} and Case \cite{Case} for the discussions of other conformally invariant boundary operators. We leave them to another paper.

 If $2m=n+1$, \eqref{eq:bi-b1} will be replaced by
\be \label{eq:bi-b2}
\pa_t \Delta^k u(x,0)=0, \quad (-1)^m\pa_t \Delta^{m-1} u(x,0)=e^{(2m-1)u} \quad \mbox{on }\pa \R^{n+1}_+
\ee
and $u$ is not necessarily positive. When $m\ge 2$, in order to have a classification theorem one has to assume that (i) $\int_{\R^n}e^{(2m-1)u(x,0)}<\infty$, (ii) $|u(x,0)|=o(|x|^2)$ near the infinity, (iii) certain growth conditions on $u(x,t)$ near the infinity.
See, for instance, Jin-Maalaoui-Martinazzi-Xiong \cite{J+} and references therein on why (i) and (ii) can not be dropped. Given (i), (ii) and (iii), one can prove a splitting theorem like Theorem \ref{thm:main-a} easily by the B\^ocher theorem (see Corollary \ref{cor:Bocher}) and Xu \cite{X}. We decide not to pursue it in this paper.

Finally, we remark that there have been many papers devoted to Liouville theorems for nonnegative solutions of nonlinear polyharmonic equations with the homogeneous Dirichlet boundary condition or homogeneous Navier boundary condition; see Reichel-Weth \cite{RW}, Lu-Wang-Zhu \cite{LWZ}, Chen-Fang-Li \cite{CFL} and references therein, where they proved that $0$ is the unique solution.

The organization of the paper is shown in the table of contents.
\medskip

\textbf{Notations:}

\begin{tabular}{ll}
  $X$& $(x,t)=(x^1,\dots,x^n,t)\subset\R^{n+1}$\\
  $B_r(X)$ &  ball with radial $r$ centered at $X$ in $\R^{n+1}$ and $B_r=B_r(0)$ \\
   $B_r^+$& $B_r\cap \R^{n+1}_+$\\
   $\pa^+ B_r^+$& $\pa B_r^+\cap \R^{n+1}_+$\\
  $D_r$ &  ball centered at the origin in $\R^{n}$, identifying $D_r=\pa B^+_r\backslash \overline{\pa^+B_r^+}$\\
  $[f]_r$ & $\fint_{\pa D_r}f\ud\sigma$, the integral average of $f$ over $\pa D_r$ \\
  $\chi_A$ & the characteristic function of the measurable set $A$ in the Euclidean spaces
\end{tabular}

\medskip

We will always assume $2m<n+1$ if it is not specified.  We will use Green's identity and its variants repeatedly:
\begin{align*}
\int_{B_1^+}(u\Delta^m\phi-\phi\Delta^m u) \,dX=&\sum_{i=1}^m\int_{\partial^+B_1^+}\left[(\Delta^{i-1}u)\frac{\partial(\Delta^{m-i}\phi)}{\partial\nu}-(\Delta^{m-i}\phi)\frac{\partial(\Delta^{i-1}u)}{\partial\nu}\right]dS\\
&-\sum_{i=1}^m\int_{D_1}\left[(\Delta^{i-1}u)\partial_t(\Delta^{m-i}\phi)-(\Delta^{m-i}\phi)\partial_t(\Delta^{i-1}u)\right]dx
\end{align*}
where $\nu$ is the outer unit normal of $\partial^+B_1^+$.

\bigskip

\section{Preliminary}
\label{sec:pre}

Let us recall that $\Delta^m$ is invariant under the $m$-Kelvin transformations
\[
u_{X_0,\lda}(X):=\left(\frac{\lda}{|X-X_0|}\right)^{n-2m+1} u\left(X_0+\frac{\lda^2 (X-X_0)}{|X-X_0|^2}\right),
\]
where $2m<n+1$, $X_0\in \R^{n+1}$ and $\lda>0$. Namely, if $u\in C^{2m}(\R^{n+1})$ then there holds
\be\label{eq:k-invariant}
\Delta^m u_{X_0,\lda}(X)=\left(\frac{\lda}{|X-X_0|}\right)^{n+2m-1}\Delta^m u\left(X_0+\frac{\lda^2 (X-X_0)}{|X-X_0|^2}\right) \quad \mbox{for }X\neq X_0.
\ee

There are various of boundary conditions for the polyharmonic equation, see Agmon-Douglis-Nirenberg \cite{ADN} or Gazzola-Grunau-Sweers \cite{GGS}. For the later use, we only consider two of them. One is like the Dirichlet condition and the other is a Neumann condition. We will be concerned with bounds of singular integrals involving the Poisson kernel and Neumann function, respectively. These bounds will play important roles in the proof of the main theorem.

\subsection{Poisson kernel for a Dirichlet problem}

Let us consider  the boundary value problem
\begin{align} \label{eq:DVP}
\begin{cases}
\Delta^m v(x,t)=0&\text{in }\R^{n+1}_+,\\
v(x,0)=f(x)&\text{on }\pa\,\R^{n+1}_+,\\
\pa_t\Delta^{k}v(x,0)=0&\text{on }\pa\,\R^{n+1}_+,
\end{cases}
\end{align}
where $f$ is a smooth bounded function in $\R^n$, and $k=0,\dots, m-2$ (if $m=1$, then we do not have this boundary condition).  Let
\[
\mathcal{P}_m(x,t)= \beta(n,m)\frac{t^{2m-1} }{(|x|^2+t^2)^{\frac{n+2m-1}{2}}},
\]
where $\beta(n,m)=\pi^{-\frac{n}{2}}\Gamma(\frac{n+2m-1}{2})/\Gamma(m-\frac 12)$ is  the normalizing constant such that
\[\int_{\R^n} \mathcal{P}_m(x,1)\,\ud x=1.\]  Note that $\mathcal{P}_1$ is the standard upper space Poisson kernel for Laplace equation.  Define
\be \label{eq:Poisson-ext}
v(x,t)=\mathcal{P}_m*f(x,t)=\beta(n,m)\int_{ \R^{n}} \frac{t^{2m-1} f(y)}{(|x-y|^2+t^2)^{\frac{n+2m-1}{2}}}\,\ud y.
\ee

\begin{lem}\label{lem:Poisson} If $f\in L^q(\R^n)$ for some $1\le q\leq \infty$, then $v$ belongs to weak-$L^{\frac{n+1}{n}}(\R^{n+1}_+)$ if $q=1$ and belongs to $L^{\frac{(n+1)q}{n}}(\R^{n+1}_+)$ if $q>1$. Moreover,
\[
\left|\{X:|\mathcal{P}_m*f(X)|>\lambda\}\right| \le c(n,m,1)\lda^{-\frac{n+1}{n}}\|f\|_{L^1(\R^n)}^{\frac{n+1}{n}}, \quad \forall ~ \lda>0,
\]
and
\[||\mathcal{P}_m*f||_{L^\frac{(n+1)q}{n}(\R^{n+1}_+)}\leq c(n,m,q)||f||_{L^q(\R^n)}, \quad \mbox{for }q>1,\]
where $c(n,m,q)>0$ is constants depending only $n,m $ and $q$.
\end{lem}
\begin{proof}
The proof by now is standard. When $m=1$, see Hang-Wang-Yan \cite{HWY}.  When $q=\infty$, it is easy to show. By the Marcinkiewicz interpolation theorem, it thus suffices to show the $q=1$ case. Without loss of generality, we may assume that $\|f\|_{L^1(\R^n)}=1$. First, note that for any $t>0$ there holds
\begin{align*}
|\mathcal{P}_m*f(x,t)|\le \beta(n,m)t^{-n}.
\end{align*}
In addition,  for any number $a>0$,
\begin{align*}
&\int_{\R^{n+1}_+\cap\{0<t<a\}}|\mathcal{P}_m*f(x,t)|\,\ud x \ud t\\\leq& \int_{\R^n}|f(y)|\ud y\int_0^a \int_{\R^n}\frac{\beta(n,m)t^{2m-1} }{(|x-y|^2+t^2)^{\frac{n+2m-1}{2}}}\,\ud x\ud t = a.
\end{align*}
It follows that for any $\lda>0$
\begin{align*}
&\left|\{(x,t):|\mathcal{P}_m*f(x,t)|>\lambda\}\right|\\=&\left|\{(x,t): 0<t<\beta(n,m)^{\frac{1}{n}}\lambda^{-\frac{1}{n}},|\mathcal{P}_m*f(x,t)|>\lambda\}\right|\\
\leq& \frac{1}{\lambda}\int_{\R^{n+1}_+\cap\{0<t<\beta(n,m)^{\frac{1}{n}} \lambda^{-\frac{1}{n}}\}}|\mathcal{P}_m*f|\, \ud x \ud t\\ \leq& \beta(n,m)^{\frac{1}{n}}\lambda^{-\frac{n+1}{n}}.
\end{align*}
Therefore, we complete the proof.
\end{proof}
\begin{lem}\label{lem:Poisson-ext} Suppose that $f$ is a smooth function in $L^q(\R^n)$ for some $q\ge 1$. Then $v$ defined by \eqref{eq:Poisson-ext} is smooth and satisfies \eqref{eq:DVP}.

\end{lem}

\begin{proof} The smoothness of $v(x,t)$ is easy and we omit the details. Note that $\mathcal{P}_m(x-y,t)$ is the Kelvin transform of $\beta(n,m)t^{2m-1}$ with respect to $X_0=(y,0)$ and $\lda=1$. It follows that $\Delta^m_{x,t} \mathcal{P}_m(x-y,t)=\beta(n,m)|X-X_0|^{-(n+2m-1)} \Delta^m_{x,t} t^{2m-1}=0$ for any $x\in \R^n$ and $t>0$. Therefore, $v$ satisfies the first equation of \eqref{eq:DVP}.

Next, let $\eta\ge 0$ be a cutoff function satisfying $\eta=1$ in $D_{1/2}$ and $\eta=0$ in $\R^n\setminus D_2$, and denote $\eta_{x_0}(x)=\eta(x-x_0)$ for any $x_0\in \R^n$. Let $v_1=\mathcal{P}_m * (f\eta_{x_0})$ and $v_2= \mathcal{P}_m * (f(1-\eta_{x_0}))$, then $v=v_1+v_2$. Clearly, \[
\lim_{(x,t)\to (x_0,0)}v_2(x,t)\to 0.
\]
By the change of variables $x-y=tz$, we see that
\[
v_1(x,t)= \beta(n,m)\int_{\R^n}\frac{(f\eta_{x_0})(x-tz)}{(|z|^2+1)^{\frac{n+2m-1}{2}}}\,\ud z.
\]
Sending  $t\to 0$, by Lebesgue dominated convergence theorem we obtain
\[
v_1(x,t)\to f(x_0) \quad \text{when } (x,t)\to (x_0,0).
\]Hence, by the arbitrary choice of $x_0$, we verified the second line of \eqref{eq:DVP}.

Finally, for any $0\le k\le m-2$, note that $\Delta^k t^{2m-1}=(2m-1)\cdots (2m-2k)t^{2m-1-2k}$ with $2m-1-2k\ge 2$. It follows that
\[
\lim_{(x,t)\to (x_0,0)}\pa_t \Delta^k v_2(x,t)= 0.
\]
Making use of $k\le m-2$ and Lebesgue dominated convergence theorem, we see that  as $t\to 0$,
\begin{align*}
&\pa_t\Delta ^k v_1(x,t)\\&= \beta(n,m)\int_{\R^n}\frac{\pa_t\sum_{j=0}^kC(j)\Delta^{k-j}_x\partial_t^{2j}(f\eta_{x_0})(x-tz)}{(|z|^2+1)^{\frac{n+2m-1}{2}}}\,\ud z\\&
=\beta(n,m)\int_{\R^n}\frac{\sum_{j=0}^kC(j)\Delta^{k-j}_x\sum_{|\alpha|=2j+1}C(\alpha)D^\alpha_x(f\eta_{x_0})(x-tz)z_1^{\alpha_1}\cdots z_{n}^{\alpha_n}}{(|z|^2+1)^{\frac{n+2m-1}{2}}}\,\ud z\\&
\to - \beta(n,m)\int_{\R^n}\frac{\sum_{j=0}^{k}C(j)\sum_{|\alpha|=2j+1}C(\alpha)\Delta_x^{k-j}D_x^\alpha(f\eta_{x_0})(x)z_1^{\alpha_1}\cdots z_{n}^{\alpha_n}}{(|z|^2+1)^{\frac{n+2m-1}{2}}}\,\ud z=0
\end{align*}
where $C(j)$ and $C(\alpha)$ are some binomial constants and  we used the oddness of the integrand in the last equality.

Therefore, we complete the proof.
\end{proof}

\begin{rem} If $f\in L^1(\R^n)$ is smooth in an open set $\om\in \R^n$ for instance. From the proof of Lemma \ref{lem:Poisson-ext}, we see that $v$  will satisfy boundary conditions of \eqref{eq:DVP} on $\om$ pointwisely.

\end{rem}

Next lemma shows the convolution with $\mathcal{P}_m$ commutes with $m-$Kelvin transformation.

\begin{lem}\label{lem:commute}
Suppose $f_{x_0,\lda}(x):=|x|^{2m-1-n}f(x_0+\frac{\lda(x-x_0)}{|x-x_0|^2})\in L^1(\R^n)$ for some $x_0\in \R^n$ and $\lda>0$. Let $X_0=(x_0,0)$.  Then $v_{X_0,\lda}=\mathcal{P}_m*f_{x_0,\lda}$.
\end{lem}
\begin{proof} We only verify the case $x_0=0$ and $\lda=1$, because the other situations are similar. Since $f_{0,1}\in L^1(\R^n)$, $\mathcal{P}_m*f_{0,1}$ is well-defined.
 By direct computations,
\begin{align*}
v_{0,1}(X)&=|X|^{2m-1-n}\mathcal{P}_m*f\left(\frac{x}{|X|^2},\frac{t}{|X|^2}\right)\\
&=\beta(n,m)|X|^{2m-1-n}\int_{\R^n}\frac{(t/|X|^2)^{2m-1}f(y)}{(|x/|X|^2-y|^2+(t/|X|^2)^2)^{\frac{n+2m-1}{2}}}\,\ud y\\
&=\beta(n,m)|X|^{-2m+1-n}\int_{\R^n}\frac{t^{2m-1}f(y)}{(|x/|X|^2-y|^2+(t/|X|^2)^2)^{\frac{n+2m-1}{2}}}\,\ud y\\
&=\beta(n,m)\int_{\R^n}\frac{t^{2m-1}|y|^{1-2m-n}f(y)}{(t^2+|y/|y|^2-x|^2)^{\frac{n+2m-1}{2}}}\,\ud y\\
&=\beta(n,m)\int_{\R^n}\frac{t^{2m-1}|z|^{2m-1-n}f(\frac{z}{|z|^2})}{(t^2+|z-x|^2)^{\frac{n+2m-1}{2}}}\,\ud z=\mathcal{P}_m*f_{0,1}(X),
\end{align*}
where in the fourth step we used the elementary equality
\[|X|^2\left(\left(\frac{t}{|X|^2}\right)^2+\left|\frac{x}{|X|^2}-y\right|^2\right)=|y|^2\left(t^2+\left|\frac{y}{|y|^2}-x\right|^2\right).\]
\end{proof}
\begin{rem}\label{rem:2.5}
Actually the proof holds whenever $\mathcal{P}_m*f_{x_0,\lambda}$ is well defined, for example $f_{x_0,\lambda}\in L^1_{loc}(\R^n)$ and bounded at infinity.
\end{rem}
\begin{lem}\label{lem:de} Let $v\in C^{2m}(\R^{n+1}_+\cup \pa\R_+^{n+1})$ be a solution of \eqref{eq:DVP}. Then for any $X_0=(x_0,0)$ and $\lda>0$, $v_{X_0,\lda}$ satisfies \eqref{eq:DVP}  with $f$ replaced by $f_{x_0,\lda}$, except the the boundary point $X_0$.

\end{lem}

\begin{proof}
It follows from direct computations.
\end{proof}

\subsection{Neumann function for a Neumann problem}

Now, we consider
\begin{align} \label{eq:NVP}
\begin{cases}
\Delta^m v(x,t)=0&\text{in }\R^{n+1}_+,\\
\pa_t\Delta^{k}v(x,0)=0&\text{on }\pa\,\R^{n+1}_+,\\
(-1)^m\pa_t \Delta^{m-1} v(x,0)=f(x)&\text{on }\pa\,\R^{n+1}_+,
\end{cases}
\end{align}
where $f$ is a smooth function belonging to $L^{q}(\R^n)$ for some $q\ge 1$, and $k=0,\dots, m-2$. Let
\[
\mathcal{N}_m(x,t)= \gamma(n,m)\frac{1}{(|x|^2+t^2)^{\frac{n-2m+1}{2}}},
\]
where $\gamma(n,m)=\pi^{\frac{n+1}{2}}\Gamma(\frac{n-2m+1}{2})/\Gamma(m)$. Define
\be \label{eq:neumann-ext}
v(x,t):= \mathcal{N}_m * f(x,t) =\gamma(n,m)\int_{\R^n}\frac{f(y)}{(t^2+|x-y|^2)^{\frac{n-2m+1}{2}}}\,\ud y.
\ee

\begin{lem}\label{lem:V} If $f\in L^1(\R^{n})$, then
$v(x,t)$ belongs to weak$-L^{\frac{n+1}{n-2m+1}}(\R_+^{n+1})$. Moreover,
\[\left|\left\{(x,t):|v(x,t)|> \lambda\right\}\right|\leq C(n,m)\lambda^{-\frac{n+1}{n-2m+1}}||f||_{L^1(\R^n)}^{\frac{n+1}{n-2m+1}}\quad \text{for every } \lambda>0,\]
where $C(n,m)>0$ is a constant depending only $n$ and $m$.
\end{lem}
\begin{proof}
 The lemma was proved by Dou-Zhu \cite{DZ} and we include a proof below for completeness and convenience of the readers.

 After scaling, assume $\int_{\R^n}f(y)dy=1$. Split $v$ as
\begin{align*}
v(x,t)&=\gamma(n,m)\left(\int_{\R^n\cap\{|x-y|\leq r\} }+\int_{\R^n\cap \{|x-y|>r\}}\right)\frac{f(y)}{(t^2+|x-y|^2)^\frac{n-2m+1}{2}}\, \ud y\\& =:v_1(x,t)+v_2(x,t),
\end{align*}
where $r$ will be fixed later. By direct computations, we have
\begin{align*}
||v_1||_{L^1(\R_+^{n+1})}&=\gamma(n,m) \int_{\R_+^{n+1}}\int_{\R^n\cap \{|x-y|\leq r\}}\frac{|f(y)|}{(t^2+|x-y|^2)^\frac{n-2m+1}{2}}\ud y\,\ud X\\
&\leq \gamma(n,m) \int_{\R^n }|f(y)|\, \ud y\int_{\R_+^{n+1}\cap B_r}\frac{1}{|X|^{n-2m+1}}\, \ud X\leq C_1r^{2m},
\end{align*}
and
\begin{align*}
|v_2|\leq C_2r^{2m-n-1},
\end{align*}
where $C_1,C_2$ are constants depending only $n$ and $m$.
Observing the inequality
\begin{align*}
\left|\left\{(x,t):|v|\geq 2\lambda\right\}\right|\leq \left|\left\{(x,t):|v_1|\geq \lambda\right\}\right|+\left|\left\{(x,t):|v_2|\geq \lambda\right\}\right|,
\end{align*}
one can choose $r$ as $C_2r^{2m-n-1}=\lambda$, then
$\left|\left\{(x,t):|v_2|\geq \lambda\right\}\right|=0$. Thus
\begin{align*}
\left|\left\{(x,t):|v|\geq 2\lambda\right\}\right|&\leq \left|\left\{(x,t):|v_1|\geq \lambda\right\}\right|\leq C\frac{1}{\lambda}||v_1||_{L^1(\R_+^{n+1})}\\&\leq C\frac{r^{2m}}{\lambda}=C\lambda^{-\frac{n+1}{n-2m+1}}.
\end{align*}
By scaling, we complete the proof of the lemma.
\end{proof}

We refer to Dou-Zhu \cite{DZ} for strong type bounds for the convolution operator involving the Neumann function.

\begin{lem}\label{lem:Poisson-neu} Suppose that $f$ is a smooth function  belonging to $L^{q}(\R^n)$ for some $q\ge 1$. Then $v$ defined by \eqref{eq:neumann-ext} is smooth and satisfies \eqref{eq:NVP}.
 \end{lem}

 \begin{proof}  The smoothness and the first two lines of  \eqref{eq:NVP} are easy to show. For the last boundary condition, observe that
 \[\Delta^k|X|^{2m-n-1}=(2m-n-1)\cdots (2m-n+1-2k)(2m-2)\cdots(2m-2k)|X|^{2m-n-1-2k}\]
for any $k\geq 1$. It follows that
\begin{align*}
\pa_t \Delta^{m-1}v(x,t)
=&(2m-n-1)\cdots(1-n)(2m-2)\cdots 2\gamma(n,m)\int_{\R^n}\frac{tf(y)}{(|x-y|^2+t^2)^{\frac{n+1}{2}}}\ud y\\
=&(-1)^m2^{2m}\frac{\Gamma(\frac{n+1}{2})}{\Gamma(\frac{n-2m+1}{2})}\Gamma(m)\gamma(n,m)\int_{\R^n}\frac{tf(y)}{(|x-y|^2+t^2)^{\frac{n+1}{2}}}\ud y
\end{align*}
therefore
\[\pa_t \Delta^{m-1}v(x,0)=(-1)^mf(x).\]
This verifies the last boundary condition.

Therefore, we complete the proof.
 \end{proof}
\begin{lem}\label{lem:neuman_conv}
Suppose $f_{x_0,\lda}(x):=|x|^{2m-1-n}f(x_0+\frac{\lda(x-x_0)}{|x-x_0|^2})\in L^1(\R^n)$ for some $x_0\in \R^n$ and $\lda>0$. Let $X_0=(x_0,0)$.  Then $v_{X_0,\lda}=\mathcal{N}_m*(|x|^{-2(2m-1)}f_{x_0,\lda})$.
\end{lem}
\begin{proof} It is similar to the proof of Lemma \ref{lem:commute}, thus we omit the details. Same as the remark \ref{rem:2.5}, the proof holds whenever $\mathcal{N}_m*(|x|^{-2(2m-1)}f_{x_0,\lda})$ is well defined.
\end{proof}

\begin{lem}\label{lem:Kelvin-Neumann} Let $v\in C^{2m}(\R^{n+1}_+\cup \pa\R^{n+1}_+)$ be a solution of \eqref{eq:NVP}. Then for any $X_0=(x_0,0)$ and $\lda>0$, $v_{X_0,\lda}$ satisfies \eqref{eq:NVP}  with $f(x)$ replaced by $|x|^{-2(2m-1)}f_{x_0,\lda}(x)$, except the the boundary point $X_0$.

\end{lem}

\begin{proof}
It follows from direct computations.
\end{proof}
\section{Polyharmonic functions with homogeneous boundary data}

\subsection{Extensions of Liouville theorem}

It is well-known that every nonnegative  solution of
\[
\begin{cases}
\Delta u(x,t)=0& \quad \mbox{in }\R^{n+1}_+,\\
u=0&\quad \mbox{on }\pa \R^{n+1}_+,
\end{cases}
\]
has to equal $a t$ for some $a\ge 0$. A simple proof of this result is based on the boundary Harnack inequality. In this subsection, we extend this result to polyharmonic functions with homogeneous boundary conditions, for which we don't have a boundary Harnack inequality.

\begin{prop}\label{prop:Ch-sing}
Let $ u \in C^{2m}(\R_+^{n+1}\cup \pa \R_+^{n+1} )$ be a solution of
\be \label{eq:ch5}
\begin{cases}
\Delta^{m} u(x,t)=0 \quad &\text{in }\R^{n+1}_+,\\
u(x,0)=0&\text{on }\pa\,\R^{n+1}_+,\\
\pa_t\Delta^{k}u(x,0)=0 \quad &\text{on }\pa\,\R^{n+1}_+, \quad k=0,1,\cdots, m-2.
\end{cases}
\ee
Suppose that ${u^*}(X)\in L^1(B_1^+)$, where $
u^*:=u_{0,1}
$ is the $m-$Kelvin transform of $u$ with respect to $X_0=0$ and $\lda=1$.  Then
\be \label{eq:poly}
u(x,t)=\sum_{k=1}^{m-1} t^{2k}P_{2k}(x)+c_{0}t^{2m-1},
\ee
where $P_{2k}(x)$ are polynomials w.\,r.\,t. $x$ of degree $\le 2m-1-2k$.

In addition if we assume ${u^*}(X)\geq g(X)$ for some $g\in L^{\frac{n+1}{n}}(B_1^+)$, then $c_0\geq 0$,  and $\text{deg}\,P_{2k}\leq 2m-2-2k$. In particular $P_{2(m-1)}$ must be a constant.
\end{prop}

\begin{proof}
For any $r>0$,  let $v(X)={u^*}(rX)$. Then $v(X)$ satisfies \eqref{eq:ch5} pointwisely except the origin. By the standard estimates for solutions of linear elliptic PDEs, we have
\be\label{eq:17}
\|v\|_{L^{\infty}(B_{5/4}^+\setminus B_{3/4}^+ )}\le C(m,n)\|v\|_{L^1(B_{3/2}^+\setminus B_{1/2}^+)}.
\ee
See \cite{ADN} or Theorem 2.20 of \cite{GGS} precisely.
Notice that
\[
\|v\|_{L^1 (B_{3/2}^+\setminus B_{1/2}^+)}=\frac{1}{r^{n+1}}\|{u^*}\|_{L^1(B^+_{3r/2}\setminus B^+_{r/2})}=o(r^{-(n+1)})\quad \text{ as }r\to 0.
\]
Together with \eqref{eq:17}, the above inequality yields
\[
|{u^*}(X)|=o(|X|^{-(n+1)})\quad \text{ as }|X|\to 0.
\]
Since $u(X)=|X|^{2m-1-n}{u^*}\left({X}/{|X|^2}\right)$, we obtain
\be \label{eq:asymp}
|u(X)|=o(|X|^{2m})\quad \text{ as }|X|\to \infty.
\ee
For every $R>0$, by the standard estimates for solutions of linear elliptic PDEs we obtain
\[
\|\nabla^{2m} u\|_{L^{\infty}(B_{R}^+)}\le CR^{-2m}\|u\|_{L^\infty(B_{2R}^+)},
\]
where $C>0$ is independent of $R$. Sending $R\to \infty$ and making use of \eqref{eq:asymp} we have
\[
\nabla^{2m} u\equiv 0 \quad \mbox{in }\R^{n+1}_+.
\]
It follows that $u$ is a polynomial of degree at most $2m-1$. Sorting $u$ by the degree of $t$, one can have
\[u(x,t)=\sum_{l=0}^{2m-2}t^{l}P_{l}(x)+c_0t^{2m-1}\]
where $P_l(x)$ is a polynomial of $x$ with degree $\leq 2m-1-l$.
The boundary conditions of $u$ imply $P_l\equiv 0$ when $l\leq 2m-2$ and is odd. Indeed, suppose the contrary and let $P_{l_0}\neq 0$ of the least odd order ${l_0}$.  Set $k_0=(l_0-1)/2\leq m-2$ which is an integer. Then
\[u(x,t)=\sum_{k=1}^{k_0}t^{2k}P_{2k}(x)+t^{l_0}P_{l_0}(x)+\sum_{l=l_0}^{2m-2}t^lP_l(x).\]
Applying $\pa_t\Delta^{k_0}$ to $u$, then $\pa_t\Delta^{k_0}(t^{2k}P_{2k}(x))(x,0)=0$ and $\pa_t\Delta^{k_0}(t^lP_{l}(x))(x,0)=0$ for
any $l> l_0$. Since $\pa_t\Delta^{k_0} u(x,0)=0$,
\[0=\pa_t\Delta^{k_0}(t^{l_0}P_{l_0}(x))(x,0)=l_0 !P_{l_0}(x).\]
Hence, we proved the claim. It follows that
\[u(x,t)=\sum_{k=1}^{m-1} t^{2k}P_{2k}(x)+c_{0}t^{2m-1}.\]

If ${u^*}\geq g$ for some $g$ as stated in the theorem.
For any polynomial $P$ with  $\text{deg}\,P<2m-1-2k$, we have
\begin{align}\label{eq:lowerfact}
(t^{2k}P(x))^*=|X|^{2m-1-n}\left(\frac{t}{|X|^2}\right)^{2k}P\left(\frac{x}{|X|^2}\right)=O(|X|^{1-n}) \text{ as }|X|\to 0\end{align}
which means $(t^{2k}P(x))^*\in L^{\frac{n+1}{n}}(B_1^+)$. Absorbing all these lower order terms of $P_{2k}$ to $g$ and collecting all the leading terms of each $P_{2k}$  to be $\tilde u$, we have
\begin{align*}
\tilde u^*=&\sum_{k=1}^{m-1}|X|^{2m-1-n}\left(\frac{t}{|X|^2}\right)^{2k}\tilde{P}_{2k}\left(\frac{x}{|X|^2}\right)+c_0|X|^{2m-1-n}\left(\frac{t}{|X|^2}\right)^{2m-1}\geq \tilde g
\end{align*}
where $\tilde{P}_{2k}$ are homogeneous polynomial in $x$ with degree equals to $2m-1-2k$ or $\tilde{P}_{2k}\equiv 0$. By the homogeneity,
\begin{align*}
\tilde u^*=&|X|^{1-2m-n}t^{2m-1}\left(\sum_{k=1}^{m-1}\tilde P_{2k}\left(\frac{x}{t}\right)+c_0\right).
\end{align*}

Note that $\tilde{P}_{2k}$ is a homogeneous polynomial of odd degree and thus $\tilde{P}_{2k}(-y)=-\tilde{P}_{2k}(y)$. Therefore if some $\tilde{P}_{2k}$ is not zero, then $\sum_{k=1}^{m-1} \tilde P_{2k}(y)+c_0$ will be negative on some open set $A\subset \R^n$ with measure $|A|=\infty$. This leads to $\tilde u^*<0$ on set $A^+=\{(x,t)\in B_1^+|x/t\in A\}$ with $|A^+|>0$.
While on this set, $\tilde u^*\not\in L^{\frac{n+1}{n}}$, which will violate the fact $\tilde u\geq \tilde g$ with $\tilde g\in  L^{\frac{n+1}{n}}(B_1^+)$. Indeed, take a bounded subset $E$ of $A$ with $|E|>0$, notice when $t_0>0$ small enough, we have $\{(tx,t):x\in E,0<t<t_0\}\subset A^+$, then
\begin{align*}
\int\int_{A^+}|\tilde u|^\frac{n+1}{n}\ud x\ud t\geq&\int_0^{t_0}\int_{tE}|\tilde u^*|^{\frac{n+1}{n}}\ud x\ud t\\
=&\int_0^{t_0}\int_{tE}\left[|X|^{1-2m-n}t^{2m-1}\left|\sum_{k=1}^m \tilde P_{2k}(x/t)+c_0\right|\right]^{\frac{n+1}{n}}\ud x\ud t\\
=&\int_0^{t_0}t^{-1}\int_{E}\left[(|y|^2+1)^{\frac{1-2m-n}{2}}\left|\sum_{k=1}^m\tilde P_{2k}(y)+c_0\right|\right]^{\frac{n+1}{n}}\ud y\ud t\\
>&c\int_0^{t_0}t^{-1}\ud t=\infty\text{ for some }c>0,
\end{align*}
where we have changed variable $x=ty$.
Therefore, $\tilde P_{2k}\equiv 0$ for $1\leq k\leq m-1$ and $c_0\geq 0$.

We complete the proof of the proposition.
\end{proof}

\begin{thm}\label{thm:positive_homo}
Let  $0\le  u \in C^{2m}(\R_+^{n+1}\cup \pa \R_+^{n+1} )$ be a solution of \eqref{eq:ch5}. Then
\be \label{eq:poly-1}
u(x,t)=\sum_{k=1}^{m-1} t^{2k}P_{2k}(x)+c_{0}t^{2m-1},
\ee
where $P_{2k}(x)$ are polynomials w.\,r.\,t. $x$ of degree $\le 2m-2-2k$, and $c_0\ge 0$.

\end{thm}

\begin{proof} By Proposition \ref{prop:Ch-sing}, it suffices to show $u^*\in L^1(B_1^+)$. Note that  $u^*$ satisfies \eqref{eq:ch5} except the origin. Define
\[
\eta_\va(t)=\begin{cases}
\frac{1}{2m!}(t-\va)^{2m}& \quad \mbox{for } t\ge \va,\\
0 & \quad \mbox{for }t<\va.
\end{cases}
\]
Since ${u^*}$ is smooth in on $\pa^+B_1^+$ and $\eta(t)\in C^{2m-1,1}$, multiplying both sides of the polyharmonic equation of ${u^*}$ and using Green's identity we have
\[
\int_{B_1^+\cap \{t>\va\}} {u^*}(X)\,dX \le C,
\]
where $C$ is independent of $\va$. Sending $\va\to 0$ and using ${u^*}\ge 0$, by Lebesgue's monotone convergence theorem we have ${u^*}\in L^1(B_1^+)$.

Therefore, we complete the proof.
\end{proof}
\subsection{Extensions of B\^ocher theorem}
In this subsection, we will give some extensions of the classical B\^ocher theorem which says that every nonnegative harmonic function in the punctured unit ball is decomposed to the fundamental solution multiplied by a constant plus a harmonic function cross the origin. Let
\[
\Phi(X)=c(m,n)\begin{cases}
|X|^{2m-n-1} &\quad \mbox{if }2m<n+1,\\
\ln |X|& \quad \mbox{if }2m=n+1,
\end{cases}
\]
be the fundamental solution of $(-\Delta)^m$,
where $c(m,n)$ is a normalization constant such that $(-\Delta)^{m} \Phi(X)=\delta_0$.

\begin{thm}\label{thm:whole-a} Let $u\in C^{2m}( B_1\setminus \{0\})$ be a solution of $(-\Delta) ^{2m} u=0$ in $B_1\setminus \{0\}\subset \R^{n+1}$.
Suppose $u\in L^1(B_1)$, then
\[
u(X)= h(X)+\sum_{|\al|\le 2m-1}c_\al D^\al \Phi(X) \quad \mbox{in }B_1,
\]
where $\al=(\al_1,\dots, \al_{n+1})\in \mathbb{N}^{n+1}$ is multi-index, $c_{\al}$ are constants, and $h$ is a smooth solution of $(-\Delta) ^{2m} h=0$ in $B_1$.
If in addition assume $u\ge g$ for some $g$ belonging to weak-$L^{\frac{n+1}{n-1}}(B_1)$, then $c_\al=0$ for $|\alpha|=2m-1$.

\end{thm}

\begin{proof} The first part of theorem was proved by Futamura-Kishi-Mizuta \cite{FKM}. For the second part, noticing when $|\al|=2m-1$, $D^\al\Phi(X)$ is homogeneous and has negative part comparable to $|X|^{-n}$, which does not belong to weak-$L^{\frac{n+1}{n-1}}(B_1)$. So $c_\al=0$ for such $\al$.
\end{proof}

We refer to Futamura-Kishi-Mizuta \cite{FKM}, Ghergu-Moradifam-Taliaferro \cite{GMT} and references therein for related works on  B\^ocher's theorem of higher order equations.

\begin{cor}  \label{cor:Bocher} Let $u\in C^{2m}(\bar B_1^+\setminus \{0\})$ be a solution of \begin{align}\label{eq:L1_property}
\begin{cases}
(-\Delta)^m u=0 \quad &\mbox{in }B_1^+,\\
\pa_t u=\pa_t \Delta u=\cdots =\pa_t \Delta^{m-1}u=0\quad& \mbox{on }D_1\setminus \{0\}.
\end{cases}
\end{align}
Suppose that $u\in L^1(B_1^+)$ and
$u\geq g$  for some $g$ belonging to weak-$L^{\frac{n+1}{n-1}}(B_1^+)$,
then
\[
u(X)= h(X)+\sum_{|\al|\le 2m-2}c_{\al} D^{\al} \Phi(X),
\]
where $\al=(\al_1,\dots, \al_{n},\al_{n+1})\in \mathbb{N}^{n+1}$ with $\al_{n+1}$ being even, and $h(X)$ satisfies
\begin{align}\label{eq:prop_h}
\begin{cases}
(-\Delta)^m h=0 \quad &\mbox{in }B_1^+,\\
\pa_t h=\pa_t \Delta h=\cdots =\pa_t \Delta^{m-1}h=0\quad &\mbox{on }D_1.
\end{cases}
\end{align}
\end{cor}

\begin{proof} Let $u(x,t)=u(x,-t)$ and $g(x,t)=g(x,-t)$ for $t<0$. We abuse the notation to denote these two new functions still as $u$ and $g$, respectively. From the boundary condition and  regularity theory for Poisson  equation, we have $(-\Delta)^{m-1} u, (-\Delta)^{m-2} u, \dots , u$ are smooth in $B_1\setminus \{0\}$. Consequently, Theorem \ref{thm:whole-a} implies the decomposition of $u$. The boundary condition actually implies we can only have $D^\al\Phi$ in the decomposition with $\al_{n+1}$ of $\al=(\al_1,\dots, \al_{n},\al_{n+1})$ is even, see the proof of the last statement of Proposition \ref{prop:Ch-sing}.

Therefore, we complete the proof.
\end{proof}

B\^ocher theorem for positive harmonic functions can be viewed as a stronger version of Liouville theorem. Indeed,

\begin{cor}\label{cor:Bocher1} Let $u\in C^{2m}(\R^{n+1}_+\cup \{\partial\,\R^{n+1}_+\setminus \{0\}\})$ be a solution of \begin{align}\label{eq:Bocher1}
\begin{cases}
(-\Delta)^m u=0 \quad &\mbox{in }\R^{n+1}_+,\\
\pa_t u=\pa_t \Delta u=\cdots =\pa_t \Delta^{m-1}u=0\quad& \mbox{on }\pa \R^{n+1}_+\setminus \{0\}.
\end{cases}
\end{align}
Suppose that $u\in L^1(B_1^+)$ and
$u\geq g$  for some $g$ belonging to weak-$L^{\frac{n+1}{n-1}}(B_1^+)$, and $\displaystyle\lim_{|X|\to \infty} u(X)=0$.
Then
\[
u(X)=\sum_{|\al|\le 2m-2}c_{\al} D^{\al} \Phi(X) \quad \forall~X\in \R^{n+1}_+,
\]
where $\al=(\al_1,\dots, \al_{n},\al_{n+1})\in \mathbb{N}^{n+1}$ with $\al_{n+1}$ being even.
\end{cor}

\begin{proof} Applying Corollary \ref{cor:Bocher} with $B_1^+$ replaced by consecutively large half balls, we have
\[
u(X)=h(X)+\sum_{|\al|\leq 2m-2}c_\al D^\al\Phi(X) \quad \forall~X\in \R^{n+1}_+,
\]
with each $\al$'s $\al_{n+1}$ even. Since $|\al|\le 2m-2$,
\[\lim_{|X|\to \infty}|h(X)|\le \lim_{|X|\to \infty}|u(X)|+\lim_{|X|\to \infty}\left|\sum_{|\al|\leq 2m-2}c_\al D^\al\Phi(X)\right|=0.\]
By the \eqref{eq:prop_h},  extending $h$ to lower half plane one can get a smooth polyharmonic function on $\R^{n+1}$ which is bounded and converges to $0$ as $|X|\to \infty$. By the interior estimates for solutions of linear elliptic PDEs,  one can easily obtain that $h\equiv 0$.
Therefore, we complete the proof.
\end{proof}

The method of proof of  Proposition \ref{prop:Ch-sing} can give a direct proof of Corollary \ref{cor:Bocher1}. Corollary \ref{cor:Bocher} is of independent interest and will be useful in study of local analysis of solutions of the nonlinear problem.

\section{Isolated singularity for nonlinear boundary data}

Now let us go back to the nonlinear boundary problems we want to study.
Suppose $0\le u\in C^{2m}(\R^{n+1}_+\cup \pa \R^{n+1}_+)$ be a solution of \eqref{eq:thma-1} with $1< p\le \frac{n+(2m-1)}{n-(2m-1)}$. Then, {by Lemma \ref{lem:Kelvin-Neumann}}, $u^*=u_{0,1}$ satisfies
\be \label{eq:thma-1'}
\begin{cases}
\Delta^{m} {u^*}(x,t)=0 \quad &\mbox{in }\R^{n+1}_+,\\
\pa_t{u^*}=\pa_t\Delta u^*=\cdots=\pa_t\Delta^{m-2}u^*(x,0)=0, &\mbox{on }\pa \R^{n+1}_+\setminus \{0\},
\\ (-1)^m\pa_t \Delta^{m-1} {u^*}(x,0)= |x|^{-\tau}{u^*}^{p} \quad &\mbox{on }\pa \R^{n+1}_+\setminus \{0\},
\end{cases}
\ee
where $\tau=[n+(2m-1)]-p[n-(2m-1)]\geq 0$. The goal of this section is to show:
\begin{prop}\label{prop:integrability} Let $u^*$ be as above. If either one of the two items holds
\begin{enumerate}
\item [(1)]$m$ is odd;
\item [(2)]$m$ is even and $u(X)=o(|X|^{2m-1})$ as $|X|\to \infty$,
\end{enumerate} then
\begin{align}\label{eq:goal}
\int_{D_1}|x|^{-\tau}u^*(x,0)^p\ud x<\infty.
\end{align}
\end{prop}

Let us start from basic properties of $u^*$.

\begin{lem}\label{lemma:babyversion} Let ${u^*}$ be a nonnegative solution of \eqref{eq:thma-1'}. Then
\begin{itemize}
\item[(i)] ${u^*}\in L^1(B_1^+)$,
\item[(ii)] $\int_{D_1} |x|^{2m-\tau}{u^*}(x,0)^pdx<\infty$,
\item[(iii)] If $p>1$, then $\int_{D_1}{u^*}(x,0)^{s}dx<\infty$ for some $s>1$.
\end{itemize}

\end{lem}

\begin{proof}
(i) ${u^*}\in L^1(B_1^+)$  was shown in the proof of Theorem \ref{thm:positive_homo}.

(ii) Let $r=|X|$,  and construct a smooth radial function $\xi_\va$ such that $\Delta ^m \xi_\va(r)=\chi_{\{r>\va\}}(r)$ for given $\va>0$, and $\xi_\va=0$ in $B_{\va/2}$. It is easy to show $\xi_\va\to \frac{1}{C(m,n)} r^{2m}$ in $C^{0}$, where $C(m,n)=\Delta^m r^{2m}>0$. Since $\xi_\va$ is radially symmetric, then $\pa_t\Delta^k \xi_\va(x,0)=0$ for any $k\geq 0$. Noticing that $\pa_t \Delta^k {u^*}$ vanishes for $k=0,\dots,m-2$ and using $\xi_\va$ as a test function in Green's identity,  we obtain
\[
\int_{D_1} \xi_\va(|x|) |x|^{-\tau}{u^*}(x,0)^p\,\ud x\le (-1)^m \int_{B_1^+} {u^*}(X)\chi_{\{r>\va\}}(|X|)\ud X+C.
\]
By item (i) and sending $\va\to 0$, then  $|x|^{2m-\tau}{u^*}(x,0)^p\in L^1(D_1)$.

(iii) By the definition of $\tau$, it is easy to check
\[p>\frac{2m-\tau}{n}+1.\]
Choosing $b$ such that
\[ \max\left\{\frac{2m-\tau}{n}+1,1\right\}< b< p,\]
then from H\"older's inequality
\[\int_{D_1}{u^*}(x,0)^{\frac{p}{b}}\ud x\leq\left(\int_{D_1}|x|^{2m-\tau}{u^*(x,0)}^p\ud x\right)^\frac{1}{b}\left(\int_{D_1}|x|^{-\frac{2m-\tau}{b-1}}\ud x\right)^{1-\frac{1}{b}}.\]
Noticing $(2m-\tau)/(b-1)<n$, it yields ${u^*}(x,0)\in L^s(D_1)$ for $s=p/b>1$.

Therefore, the lemma is proved.
\end{proof}

{
Since $u^*(x,0)\in L^1(D_1)$ and $u^*(x,0)\in L^\infty(\R^{n}\setminus D_1)$, then
\be\label{eq:v-star}
v^*:=\mathcal{P}_m*u^*
\ee is well-defined. }

\begin{prop} \label{pro:reduction} Let $v^*$ be in \eqref{eq:v-star}. Then we have ${v^*}\in L^{\frac{(n+1)}{n}}(B_1^+)$ and
\begin{align}
\begin{cases}\label{eq:v_kelvin}
\Delta^{m} {v^*}(x,t)=0 \quad &\text{in } \R^{n+1}_+,\\
\pa_t{v^*}=\pa_t\Delta v^*=\cdots=\pa_t\Delta^{m-2} v^*(x,0)=0 \quad &\text{on }\pa \,\R^{n+1}_+\setminus \{0\},\\
(-1)^m\pa_t \Delta^{m-1} {v^*}(x,0)= |x|^{-\tau}{v^*}^{p}-c_0{(-1)^m} |x|^{-(2m-1+n)} \quad &\text{on }\pa\, \R^{n+1}_+\setminus \{0\},
\end{cases}
\end{align}
where $c_0\ge 0$ is a constant.
\end{prop}

\begin{proof}
Decompose ${u^*}(x,0)={u}^*_1(x,0)+{u}^*_2(x,0)$ for $x\in \R^n$, where $u^*_1(x,0)={u^*}(x,0)\chi_{D_1}(x)$ and $\chi_{D_1}$ is the characteristic function of $D_1$. Then ${v^*}=v^*_1+v^*_2$ with $v_1^*$ and $v_2^*$ are given by the corresponding Poisson type convolutions of $u_1^*(x,0)$ and $u_2^*(x,0)$ as in \eqref{eq:v-star}, respectively.

Since ${u^*}\in L^{s}(D_1)$ for some $s>1$ by Lemma \ref{lemma:babyversion},  we have $v^*_1\in {L^{\frac{(n+1)s}{n}}}(\R^{n+1}_+)$ by Lemma \ref{lem:Poisson}.
On the other hand, since ${u^*}(x,0)=O(|x|^{2m-1-n})$ as $x\to\infty$,  then $u^*_2(x,0)\in L^{q}(\R^n)$ for any {$q>\frac{n}{n+1-2m}$}. Using Lemma \ref{lem:Poisson} again yields $v^*_2\in L^{\bar q}(\R^{n+1}_+)$ for any {$\bar q>\frac{n+1}{n+1-2m}$.}
Restricting $v_1^*$ and $v_2^*$ in $B_1^+$ and notice that
\[\min\left\{\frac{(n+1)s}{n},\frac{n+1}{n+1-2m}\right\}>\frac{n+1}{n},\]
we proved ${v^*}\in L^{\frac{n+1}{n}}(B_1^+)$.

By Lemma \ref{lem:de}, $v^*$ satisfies the first two lines of \eqref{eq:v_kelvin}. Let $v=(v^*)_{0,1}$. By Lemma \ref{lem:commute} and the remark after it,  ${v}(x,0)={u}(x,0)$ on $\R^n$. Define $w={u}-{v}$, which satisfies \eqref{eq:ch5} in Proposition \ref{prop:Ch-sing}. $w^*\geq -v^*$ will satisfy the assumption of Proposition \ref{prop:Ch-sing}, therefore we conclude
\be \label{eq:poly1}
w(x,t)=\sum_{k=1}^{m-1} t^{2k}P_{2k}(x)+c_{0}t^{2m-1},
\ee
where $c_0\ge 0$, $P_{2k}(x)$ are polynomials w.\,r.\,t.\,$x$ of degree $\le 2m-2-2k$. Therefore,
\[\pa_t\Delta^{m-1}{v^*}=\pa_t\Delta^{m-1}{u^*}-\pa_t\Delta^{m-1}w^*,\]
Since
\begin{align*}
\pa_t\Delta^{m-1}w^*(x,0)&=c_0\pa_t\Delta^{m-1}(|X|^{1-2m-n}t^{2m-1})(x,0)\\
&=c_0(2m-1)!|x|^{-(2m-1+n)},
\end{align*}
the proposition follows immediately.
\end{proof}

Naively one may wish $c_0=0$, then $u^*$ and $v^*$ share the same equations. However, as we said in the introduction, there are special cases, for example when $m$ is even, $u^*$ will be the $m-$Kelvin transformation of $H_a(x,t)$ in \eqref{eq:specialsolution}, but $v^*\equiv a^{1/p}|X|^{2m-1-n}$, so $c_0\neq 0$. On the other hand, we will prove that under the assumptions in Proposition \ref{prop:integrability}, we have $c_0=0$. To that end, we need to analyze the symmetrization of the solutions. When applied to  radially symmetric functions  in $\R^{n+1}$ the Laplace operator $\Delta$ is expressed as
\[L=\frac{d^2}{dr^2}+\frac{n}{r}\frac{d}{d r}.\]

\begin{lem}\label{lem:ODE}
Suppose that $w\in C^{2m}(\R^{n+1}_+\cup \{\pa\,\R^{n+1}_+\setminus \{0\}\})$ satisfies
\begin{align*}
\begin{cases}
\Delta^{m} {w}(x,t)=0 \quad &\text{in } \R^{n+1}_+,\\
\pa_t{w}=\cdots=\pa_t\Delta^{m-2}w(x,0)=0 \quad &\text{on }\pa \R^{n+1}_+\setminus \{0\},\\
(-1)^m\pa_t \Delta^{m-1} {w}(x,0)= f(x) \quad &\text{on }\pa \R^{n+1}_+\setminus \{0\}.
\end{cases}
\end{align*}
Then
\begin{align}
L^{m}\bar {w}(r)&=(-1)^m\frac{\w_{n-1}}{\w_{n}}r^{-1} [f]_r,
\label{eq:symm-w}
\end{align}
where  $\bar {w}(r)= \fint_{\pa^+ B_r^+}  {w}(x,t)\,\ud S_{x,t}$ and  $[f]_{r}=\fint_{\pa D_r} f(x)\,\ud \sigma$ and $\w_n,\w_{n-1}$ are the volume constants.

\end{lem}

\begin{proof}
By the definition of $\bar {w}$, taking derivatives leads to
\[
r^{n}\bar {w}'(r)=\frac{1}{\w_n}\int_{\pa^+ B_r^+}\frac{\pa {w}}{\pa \nu}\,\ud S=-\frac{1}{\w_n} \int_{B_1^+\setminus  B_r^+}\Delta {w}\,\ud X+\frac{1}{\w_n}\int_{\pa^+ B_1^+}\frac{\pa {w}}{\pa \nu}\,\ud S,
\] where $r\in(0,1)$, $\nu$ is the outer unit normal of the boundary and we used $\pa_t {w}(x,0)=0$. It follows that
\[
L \bar {w}= \fint_{\pa^+B_{r}^+} \Delta {w}\,\ud S.
\]
Using $\pa_t \Delta^k {w}(x,0)=0$ for $k=1,\dots, m-2$ and repeating this process, we have
\be \label{eq:L1}
L^{m-1} \bar {w}=\fint_{\pa^+B_{r}^+} \Delta^{m-1} {w}\,\ud S.
\ee
By Green's identity, we have for any $0<r<1$
\begin{align*}
&\int_{\pa^+ B_1^+}\frac{\pa \Delta^{m-1} {w}}{\pa \nu}\,\ud S- \int_{\pa^+ B_r^+}\frac{\pa \Delta^{m-1} {w}}{\pa \nu}\,\ud S
-\int_{D_1\setminus D_r}\pa_t\Delta^{m-1} {w}\,\ud x\\
=& \int_{B_1^+\setminus B_r^+} \Delta^m {w}=0.
\end{align*}
Taking derivative in $r$, we have
\be\label{eq:to-to-1}
\frac{d}{dr} \int_{\pa^+ B_r^+}\frac{\pa \Delta^{m-1} {w}}{\pa \nu}\,\ud S =(-1)^m\w_{n-1} r^{n-1}[f]_{r}.
\ee
Since
\be\label{eq:to-to-2}
\frac{d}{dr}\fint_{\pa^+B_{r}^+} \Delta^{m-1} {w}\,\ud S=  \fint_{\pa^+ B_r^+} \frac{\pa \Delta^{m-1} {w}}{\pa \nu}\,\ud S,
\ee
then \eqref{eq:L1} implies
\begin{align*}
L^m\bar w(r)=&\frac{1}{r^n}\frac{d}{dr}\left(r^n\frac{d}{dr}\fint_{\pa^+ B_r^+} \Delta^{m-1} {w}\,\ud S\right)=\frac{1}{r^n}\frac{d}{dr}\left(r^n\fint_{\pa^+ B_r^+} \frac{\pa \Delta^{m-1} w}{\pa \nu}\,\ud S\right)\\
=&\frac{1}{w_nr^n}\frac{d}{dr}\int_{\pa^+ B_r^+}\frac{\pa \Delta^{m-1} {w}}{\pa \nu}\,\ud S=(-1)^m \frac{w_{n-1}}{w_n}r^{-1}[f]_r.
\end{align*}
Therefore, we complete the proof.
\end{proof}
Notice that $u^*$ satisfies \eqref{eq:thma-1'} and $v^*$ satisfies \eqref{eq:v_kelvin}.
It follows from the above lemma that:
\begin{cor}
\begin{align}
\label{eq:symm-u}
L^{m}\bar {u^*}(r)&=(-1)^m\frac{\w_{n-1}}{\w_{n}}r^{-\tau-1} [{u^*}^p]_r,\\
L^{m}\bar {v^*}(r)&=(-1)^m\frac{\w_{n-1}}{\w_{n}}\{r^{-\tau-1} [{v^*}^p]_r-c_{0}(-1)^m{r^{-n-2m}}\}.
\label{eq:symm-v}
\end{align}
\end{cor}

\begin{lem}\label{lem:c-zero} Under the assumptions in Proposition \ref{prop:integrability}, we have $c_0=0$.
\end{lem}
\begin{proof} By the ODE of $\bar{u^*}$, one can integrate $2m$ times to get
\begin{align} \label{eq:ode-a}
\bar {u^*}(r)=&a\Phi(r)+\sum_{k=2}^m \left\{b_kr^{2(m-k)-n+1}+c_kr^{2(m-k)}\right\}+ \frac{(-1)^m\w_{n-1}}{\w_n} F(r),
\end{align}
where $a,b_k,c_k$ are constants depending only on $C^{2m}$ norm of $u^*$ near $\pa^+ B_1^+$, and
\[
F(r)= \int_{r}^1 r_{2m-1}^{-n}\int_{r_{2m-1}}^1r_{2m-2}^{n}\int \cdots \int_{r_4}^1 r_3^{n}\int_{r_3}^1 r_{2}^{-n}\int_{r_{2}}^1r_1^{n} r_1^{-\tau-1}[{u^*}(x,0)^p]_{r_1}\,\ud r_1\cdots \ud r_{2m-1}\ud r.
\]
If $m$ is odd, \eqref{eq:ode-a} gives $\bar {u^*}(r)\le C b_m r^{-n+1}$ for small $r$. Similarly, $\bar {v^*}(r)\le C b_m r^{-n+1}$  for small $r$. Since $u^*$ and $v^*$ are positive, $u^*$, $v^*$ and $w^*:=u^*-v^*$ must belong to weak- $L^{\frac{n+1}{n-1}}(B_1^+)$. By \eqref{eq:poly1}, $c_0 |X|^{-(n+2m-1)}t^{2m-1}$ has to belong weak-$ L^{\frac{n+1}{n-1}}(B_1^+)$, which forces $c_0=0$.

On the other hand, if $m$ is even and $u(X)=o(|X|^{2m-1})$, by \eqref{eq:poly1} and the fact that $w=u-v\leq u$ we immediately have $c_0=0$.

In conclusion, we complete the proof.
\end{proof}

Next two lemmas can boost the regularity of $u^*$ by iteration.
\begin{lem}\label{lem:upgrade}
Under the assumptions of Proposition \ref{prop:integrability}.  If $u^*(x,0)\in L^s(D_1)$ for some $s>1$, then
\[\int_{D_1}|x|^{q-\tau}u^*(x,0)^p\ud x<\infty\]
for any $q> 2m-n-1+\frac{n}{s}$.
\end{lem}
\begin{proof}

By the proof of Proposition \ref{pro:reduction}, $v^*\in L^{\tilde s}(B^+_1)$, where
\[\tilde s=\min\left\{\frac{(n+1)s}{n}, \frac{n+1}{n-2m+1}\right\}.\]
Fix  any $q> 2m-n-1+\frac{n}{s}$.
Choose $0\le \eta(r)\in C^\infty(0,\infty)$ such that
$\eta(r)=0$ when $r<1/2$ and $\eta(r)=1$ when $r>1$ and define
\[\phi_\va(X)=\eta\left(\frac{|X|}{\va}\right)|X|^q.\]
Multiplying $v^*$ by $\phi_\va$ and using Green's identity over $B^+_1$, we have
\begin{align*}
\int_{B_1}v^* \Delta^m \phi_\va\ud X= \int_{D_1} \eta\left(\frac{|x|}{\va}\right) |x|^{q-\tau} {u^*}(x,0)^p\,\ud x+C.
\end{align*}
Sending $\va\to 0$, the first term of RHS will converge to the integral we want to bound, while the LHS will be uniformly bounded. Indeed,
 by H\"older's inequality and the radial symmetry of $\phi_\va$,
\begin{align*}
\int_{B_1^+}|v^* \Delta^m \phi_\va|\ud X &\le C\sum_{k=0}^{2m}\int_{B_1^+}v^* |X|^{q-k} \left|\frac{\ud^{2m-k}}{\ud r^{2m-k}}\eta\left(\frac{r}{\va}\right)\right|\ud X\\&
\le C \int_{B_1^+}v^* |X|^{q-2m}\ud X+ C \va^{q-2m}\int_{B_\va^+} v^* \ud X\\&
\le C(n,q)\|v^*\|_{L^{\tilde s}(B_1^+)}+C(n,q)\|v^*\|_{L^{\tilde s}(B_1^+)}\va^{q-2m+(n+1)(1-1/\tilde{s})}\\&
\le C,
\end{align*}
where we used the assumption on $q$ to give $q-2m+(n+1)(1-1/\tilde{s})> n/s-(n+1)/\tilde s\geq 0$.

Therefore, we complete the proof.
\end{proof}
\begin{lem}\label{lem:dividing}
Assume the assumptions in Proposition \ref{prop:integrability}. Then for any $1<p\leq \frac{n+2m-1}{n-2m+1}$ we have
\[\int_{D_1} |x|^{q-\tau}u^*(x,0)^p\ud x<\infty\quad\text{and}\quad u^*(x,0)\in L^p(D_1)\]
where $q>2m-n-1+\frac{n}{p}$. In particular, if $p>\frac{n}{n-2m+1}$, $q$ can achieve $0$ thus \eqref{eq:goal} holds.
\end{lem}
\begin{proof}
Let us call $u^*(x,0)$ has $(q,s)-$property if
\[\int_{D_1}|x|^{q'-\tau}u^*(x,0)^p\,\ud x<\infty\quad\forall\, q'> q\text{ and } u^*(x,0)\in L^{s'}(D_1)\quad \forall\, s'< s.\]

From Lemma \ref{lemma:babyversion} item (ii), $u^*(x,0)$ has $(q_0,s_0)-$property with $q_0=2m$, $s_0=\frac{np}{n+(2m-\tau)^+}=\frac{np}{n+(q_0-\tau)^+}>1$, where $a^+=\max\{a,0\}$ for any constant $a$. From Lemma \ref{lem:upgrade}, we have
\[\int_{D_1}|x|^{q-\tau}u^*(x,0)^p\,\ud x<\infty\quad \forall\, q> q_1=2m-n-1+\frac{n}{s_0}.\]
From this, one can repeat the proof of Lemma \ref{lemma:babyversion} item (ii) to see
\[u^*(x,0)\in L^s(D_1)\quad \forall\, s<s_1=\frac{np}{n+(q_1-\tau)^+}.\]
Therefore $u^*(x,0)$ has $(q_1,s_1)-$property. Moreover, it is easy to see $q_1<q_0$ and $s_1>s_0$. By iterating all the above steps, we have $u^*(x,0)$ has $(q_k,s_k)-$property,
\begin{align}\label{eq:qn}
q_k=2m-n-1+\frac{n}{s_{k-1}}\quad\text{and}\quad s_k=\frac{np}{n+(q_k-\tau)^+}.
\end{align}
Moreover $q_0>q_1\geq \cdots\geq q_k$ and $s_0<s_1\leq \cdots\leq s_k$.

\textbf{Claim:} There exist some $k$ finite such that $q_k\leq \tau$ and $s_k=p$.

Suppose not, then we will have an infinite many $q_k>\tau$ which are non-increasing. Suppose $\lim_{k\to \infty}q_k=a\geq \tau$, consequently \eqref{eq:qn} implies
\[a=2m-n-1+\frac{a-\tau+n}{p}=\frac{a}{p}+\frac{1-2m}{p}<a,\]
which is a contradiction. The claim is proved.

Thus after some finite steps, we will have $s_k=p$ and $q_k=2m-n-1+\frac{n}{p}=\frac{\tau-2m+1}{p}<\tau$ for some $k$ finite. Namely, $u^*(x,0)\in L^p(D_1)$ and
\[\int_{D_1}|x|^{q_k-\tau}u^*(x,0)^p\,\ud x<\infty.\]
In particular if $p>\frac{n}{n-2m+1}$, then $q_k=2m-n-1+\frac{n}{p}< 0$ and
\[\int_{D_1}|x|^{-\tau}u^*(x,0)^p\,\ud x<\infty.\]
We complete the proof.
\end{proof}

In order to prove Proposition \ref{prop:integrability} in the remaining  range $1<p\leq\frac{n}{n-2m+1}$, we need to investigate the singularity of $u^*$ near origin more precisely. The following three lemmas are devoted to that. Let us build a bridge between the boundary integral and inner integral of $v^*$.

\begin{lem}\label{lem:balance} Let $v^*$ be defined by \eqref{eq:v-star} and $\va\in [0,1)$. Then for any $r_0>0$ there exists a constant $C>0$, depending only on $m,n,\va, r_0$, $\|u^*(\cdot, 0)\|_{L^\infty(\R^n\setminus D_{2r_0})}$ and $\|u^*(\cdot,0)\|_{L^1(D_{r_0})}$, such that
\begin{align}\label{eq:balance}
\int^{2r}_r\rho^{-\va}\bar {v^*}(\rho)\,\ud \rho\leq C r^{1-n-\va}\int_{D_{r/2}} u^*(y,0)\, \ud y+C\int_{r/2}^{r_0}\rho^{-\va}[{u^*}]_\rho\, \ud \rho+Cr^{1-\va}
\end{align}
 for any $r\in (0,r_0/4)$.
\end{lem}
\begin{proof} For any $r<r_0/4$, suppose $\rho\in [r,r_0]$, then we have
\begin{align*}
\fint_{\pa^+B_\rho^+}{v^*}(x,t)\,\ud S=&\fint_{\pa^+B_\rho^+}\int_{0<|y|<r/2}\mathcal{P}_m(x-y,t){u^*}(y,0)¡¢£¬\ud y\ud S\\
&+\fint_{\pa^+B_\rho^+}\int_{r/2<|y|<2r_0}\mathcal{P}_m(x-y,t){u^*}(y,0)\, \ud y\ud S\\
&+\fint_{\pa^+B_\rho^+}\int_{2r_0<|y|}\mathcal{P}_m(x-y,t){u^*}(y,0)\, \ud y\ud S\\
:=&I_1+I_2+I_3.
\end{align*}
By direct computations,
\begin{align*}
I_1&= \int_{0<|y|<r/2}u^*(y,0)\ud y\fint_{\pa^+B_\rho^+}\mathcal{P}_m(x-y,t)\,\ud S
\\&\leq C\rho^{-n}\int_{D_{r/2}}{u^*}(y,0)\,\ud y\\
I_2&\leq \int_{r/2<|y|<2r_0}u^*(y,0)\fint_{\pa^+B_\rho^+}\frac{1}{|X-Y|^n}\,\ud S\ud y\\
I_3&\leq C.
\end{align*}
where $X=(x,t)$, $Y=(y,0)$, and $C>0$ depends only on $m,n,r_0$, $\|u^*(\cdot, 0)\|_{L^\infty(\R^n\setminus D_{2r_0})}$ and $\|u^*(\cdot,0)\|_{L^1(D_{r_0})}$.
It follows that
\begin{align*}
&\fint_{\pa^+B_\rho^+}{v^*}(x,t)\, \ud S\\
\leq &C\rho^{-n}\int_{D_{r/2}}{u^*}(y,0)\, \ud y+\int_{r/2<|y|<2r_0}u^*(y,0)\fint_{\pa^+B_\rho^+}\frac{1}{|X-Y|^n}\,\ud S\ud y+C
\end{align*}
Multiplying both sides of the above inequality by $\rho^{-\va}$ and integrating from $r$ to $2r$, we obtain
\begin{align*}
&\int_r^{2r}\rho^{-\va}\bar{v^*}(\rho)\,\ud \rho\\
\leq& Cr^{1-n-\va}\int_{D_{r/2}}{u^*}(y,0)\,\ud y+C\int_{r/2<|y|<2r_0}u^*(y,0) |y|^{1-n-\va}\,\ud y+Cr^{1-\va}\\
=&Cr^{1-n-\va}\int_{D_{r/2}}{u^*}(y,0)\,\ud y+C\int_{r/2}^{2r_0}\rho^{-\va}[{u^*}]_\rho\,\ud \rho+Cr^{1-\va},
\end{align*}
where we used the inequality
\[\int_{\R^{n+1}}\frac{1}{|X-Y|^{n}}|X|^{-n-\va}\,\ud X\leq C(n,\va)|Y|^{1-n-\va}\]
with taking $Y=(y,0)$.
\end{proof}

\begin{lem}\label{lem:vanish}
Assume the assumptions of Proposition \ref{prop:integrability}, then
\[\int_{D_r}u^*(x,0)\ud x\leq C(\hat{p})r^{\hat{p}}.\]
where $\hat p>2m-1+(n-n/p)/p$.
\end{lem}
\begin{proof}
From Lemma \ref{lem:dividing} we have
\be\label{eq:lem410-a}
\int_{D_r}|x|^{q-\tau}{u^*(x,0)}^p\ud x=C(q)<\infty,
\ee
where $q>q'=2m-n-1+\frac{n}{p}$. From H\"older's inequality, we have
\be \label{eq:lem410-b}\int_{D_r}{u^*}(x,0)\ud x\leq\left(\int_{D_r}|x|^{q-\tau}{u^*(x,0)}^p\ud x\right)^\frac{1}{p}\left(\int_{D_r}|x|^{-\frac{q-\tau}{p-1}}\ud x\right)^{1-\frac{1}{p}}.
\ee
By the definition of $\tau$, one can verify
\[(n-\frac{q-\tau}{p-1})(1-\frac 1p)< (n-\frac{q'-\tau}{p-1})(1-\frac 1p)=2m-1+\frac{n-n/p}{p}.\]
It follows that
\begin{align}\label{eq:vcase1}
\left(\int_{D_r}|x|^{-\frac{q-\tau}{p-1}}\ud x\right)^{1-\frac{1}{p}}\leq C(q)r^{(n-\frac{q-\tau}{p-1})(1-1/p)}.
\end{align}
Combining \eqref{eq:lem410-a}, \eqref{eq:lem410-b} and \eqref{eq:vcase1} together, the lemma follows immediately.
\end{proof}

\begin{lem}\label{lem:F} Under the assumptions of Proposition \ref{prop:integrability} and  $1< p\leq \frac{n}{n-2m+1}$. It is \textbf{impossible} to find small constants $r_0>0$ and $a_0>0$ such that
\be\label{eq:lem411-a} \bar {v^*}(r)\ge a_0 r^{2m-1-n}\quad \forall~ r\in (0,r_0).
\ee
\end{lem}

\begin{proof}
Suppose the contrary that there exist $r_0>0$ and $a_0>0$ such that \eqref{eq:lem411-a} holds. Clearly, we can take $r_0$ being sufficiently small.
By Lemma \ref{lem:balance} and Lemma \ref{lem:vanish},  if $r_0$ is sufficiently small, we have for $r\in (0,r_0)$
\begin{align*}
r^{2m-n-\va}&\leq C(\va)r^{2m-n+(n-n/p)/p-\frac 12\va}+C\int_{r/2}^{r_0}\rho^{-\va}[{u^*}]_\rho\ud\,\rho+Cr^{1-\va}
\end{align*}
where $\va\in (0,1)$.
Taking $\va$ sufficiently small and fix it, it follows that for $r_0$ sufficiently small and $r\in (0,r_0)$ there holds
\begin{align}\label{eq:lem411-b}
\int_{r}^{r_0}\rho^{-\va}[{u^*}]_\rho\, \ud\rho\geq
\frac{1}{C}r^{2m-n-\va}.
\end{align}
By H\"older's inequality,
\begin{align*}
\int^{r_0}_r\rho^{-\va}[u^*]_\rho \,\ud\rho &\leq C\left(\int^{r_0}_r\rho^{-\va p}[(u^*)^p]_\rho\,\ud \rho\right)^{1/p}\\& \leq C\left(\int^{r_0}_r\rho^{q-\tau+n-1}[(u^*)^p]_\rho\,\ud\rho\right)^{1/p}r^{-\frac{(q-\tau+n-1+\va p)^+}{p}}\\
&=C\left(\int_{D_{r_0}}|x|^{q-\tau}u^*(x,0)^p\,\ud x\right)^{1/p}r^{-\frac{(q-\tau+n-1+\va p)^+}{p}}\\&\le C r^{-\frac{(q-\tau+n-1+\va p)^+}{p}},
\end{align*}
where we used Lemma \ref{lem:dividing}  in the last inequality.
Together with \eqref{eq:lem411-b}, the above inequality yields
\begin{align}\label{eq:singular_u}
r^{-\frac{(q-\tau+n-1+\va p)^+}{p}}\geq
\frac{1}{C}r^{2m-n-\va} \quad \forall~r\in (0,r_0).
\end{align}
Since $1<p\leq  \frac{n}{n-2m+1}$, we have
\[q-\tau+n-1=\frac{n}{p}+p(n-2m+1)-n-1\leq n-2m,\]
and thus
\[
-\frac{(q-\tau+n-1+\va p)^+}{p}>2m-n-\va,
\]
which makes \eqref{eq:singular_u} impossible.

Therefore, we complete the proof.
\end{proof}

\begin{lem}[Dichotomy lemma]\label{lem:induction}
Suppose $\xi(r)\in C^{2m}(0,\infty)$, if there exists $c_1,\tilde{c}_1,r_1,\tilde{r}_1>0$
\[L^{m-1}\xi(r)\geq c_1 r^{1-n}\text{ for any }r\in (0,r_1)\quad  \text{ or } \quad L^{m-1} \xi(r)\leq -\tilde{c}_1r^{1-n}\text{ for any }r\in (0,\tilde r_1)\]
then there exists $c_m,\tilde c_m,r_m,\tilde{r}_m>0$ such that either
\[\xi(r)\geq c_m r^{2m-1-n}\text{ for any }r\in (0,r_m)\quad  \text{ or } \quad \xi(r)\leq -\tilde{c}_mr^{2m-1-n}\text{ for any }r\in (0,\tilde r_m).\]
\end{lem}
\begin{proof}
We will prove it by induction.
Define $\xi_k=L^k\xi$, for $k=0,1,\cdots,m-1$.

(i) Suppose $\xi_{m-1}(r)=r^{-n}(r^n\xi_{m-2}')'\le - \tilde c_1r^{1-n}$, which implies $r^{n} \xi_{m-2}'$ is decreasing. There are two cases:

\textbf{Case 1}: $\displaystyle\liminf_{r\to 0}r^{n} \xi_{m-2}'\le 0$.  Then we have
\[
r^{n} \xi_{m-2}'(r)\le -\frac{\tilde c_1}{2} r^2, \quad 0<r<\tilde r_1
\]
which yields
\begin{align}\label{eq:newton}
\xi_{m-2}(\tilde r_1)-\xi_{m-2}(r)&=\int_{r}^{\tilde r_1} \xi_{m-2}'\,\ud \rho \le -\frac{\tilde c_1}{2}\int_{r}^{\tilde r_1} \rho^{2-n}\,\ud \rho=\frac{\tilde c_1}{2(n-3)}\rho^{3-n}\Big|_r^{\tilde r_1}.
\end{align}
Therefore, there exist $c_2,r_2>0$ such that
 \begin{align}
\xi_{m-2}(r)\ge c_2r^{3-n},\quad\text{ for } 0<r<r_2<\tilde r_1.
\end{align}

 \textbf{Case 2}: There exists $\hat r>0$ and $\hat c>0$ such that
\[
r^{n} \xi_{m-2}'(r)\ge \hat c, \quad \mbox{for }0<r<\hat r<\tilde r_1.
\]
Arguing as \eqref{eq:newton}, there exist $\tilde c_2, \tilde r_2>0$ such that
\begin{align}
\xi_{m-2}(r) \le -\tilde c_2 r^{1-n}\le -\tilde c_2 r^{3-n}, \quad \mbox{for }0<r<\tilde r_2<\hat r.
\end{align}

(ii) Suppose $\xi_{m-1}\geq c_1 r^{1-n}$ happens, which implies $r^{n} \xi_{m-2}'$ is increasing as $r$ goes large. There are two cases:

\textbf{Case 1:} $\displaystyle\liminf_{r\to 0}r^{n} \xi_{m-2}'\ge 0$, then we have
\[
r^{n} \xi_{m-2}'(r)\ge \frac{ c_1}{2} r^2,\quad 0<r<r_1
\]
which yields
\[\xi_{m-2}(r_1)-\xi_{m-2}(r)= \int_{r}^{r_1} \xi_{m-2}'(\rho)\,\ud \rho \ge \frac{c_1}{2} \int^{r_1}_r \rho^{2-n}\,\ud \rho.\]
Therefore, there exist $\tilde c_2, \tilde r_2>0$,
\begin{align}
\xi_{m-2}(r) \le -\tilde c_2 r^{3-n} \quad \mbox{for }0<r<\tilde r_2<\hat r.
\end{align}

\textbf{Case 2:} There exist $\hat c,\hat r>0$ such that
\[
r^{n} \xi_{m-2}'(r)\leq -\hat c \quad \mbox{for }0<r<\hat r<r_1.
\]
Arguing as before there exist $c_2,r_2>0$ such that
\[
\xi_{m-2} \geq c_2 r^{1-n}\geq c_2r^{3-n} \quad \text{for } 0<r<r_2<\hat r.
\]

For both (i) and (ii), we reached the same conclusion
\[\xi_{m-2}\geq c_2 r^{3-n}\text{ for }r\in (0,r_2)\quad  \text{ or } \quad \xi_{m-2}\leq -\tilde{c}_2r^{3-n}\text{ for  }r\in (0,\tilde r_2).\]
Repeating this procedure, we obtain
\[\xi_{k}\geq c_{m-k} r^{2(m-k)-1-n}\text{ for  }r\in (0,r_k)\quad  \text{ or } \quad \xi_{k}\leq -\tilde{c}_{m-k}r^{2(m-k)-1-n}\text{ for  }r\in (0,\tilde r_k),\]
when $0\leq k\leq m-1$. Taking $k=0$, we complete the proof the lemma.
\end{proof}

\begin{proof}[Proof of Proposition \ref{prop:integrability}] Since $p\ge \frac{n}{n-2m+1}$ was proved in Lemma \ref{lem:dividing}, now we assume $1<p\leq  \frac{n}{n-2m+1}$.
Suppose contrary that \eqref{eq:goal} is not true, then it necessarily has
\be \label{eq:keylemma-a}
\int_{D_1\backslash D_r}|x|^{-\tau} {u^*}(x,0)^p\,\ud x=\int_{D_1\setminus D_r} |x|^{-\tau}{v^*(x,0)}^{p} \,\ud x\to \infty \quad \mbox{as }r\to 0.
\ee
 Make use of the equation of ${v^*}$ and Green's identity, we have
\begin{align*}
\int_{\pa^+ B_1^+} \frac{\pa \Delta^{m-1} v^*}{\pa \nu}\,\ud S-\int_{\pa^+ B_r^+}  \frac{\pa \Delta^{m-1} v^*}{\pa \nu}\,\ud S&=\int_{D_1\setminus D_r}  {\pa_t \Delta^{m-1} v^*(x,0)}\,\ud x\\
&=\int_{D_1\setminus D_r}(-1)^m|x|^{-\tau}{v^*(x,0)}^{p}\ud x.
\end{align*}

 If $m$ is odd, by \eqref{eq:keylemma-a} there exists  $r_0>0$ such that for all $0<r<r_0$,
\[
\fint_{\pa^+ B_r^+}\frac{\pa \Delta^{m-1} v^*}{\pa \nu}\,\ud S \ge  \frac{r^{-n}}{2} \int_{D_{1/2}\setminus D_r}|x|^{-\tau}{v^*(x,0)}^{p}\,\ud x.
\]
It follows that
\[
\fint_{\pa^+ B_{r_0}^+}\Delta^{m-1} v^*\,\ud S- \fint_{\pa^+ B_r^+}\Delta^{m-1} v^*\,\ud S\ge \frac{1}{2} \int_{r}^{r_0} \lda^{-n} \int_{D_{1/2}\setminus D_\lda}|x|^{-\tau}{v^*(x,0)}^{p}\,\ud x\ud \lda,
\]
which together with \eqref{eq:keylemma-a} yield
\be \label{eq:keylemma-b}
\fint_{\pa^+ B_r^+}\Delta^{m-1} v^*\,\ud S \le - r^{1-n}
\ee
for all $0<r<r_1<r_0$, where $r_1$ is some fixed constant. Since \eqref{eq:L1} is also true for $v^*$, then we have $L^{m-1} \bar {v^*}(r)\le - r^{1-n} $.

If $m$ is even,  by \eqref{eq:keylemma-a} there exists  $r_0>0$ such that for all $0<r<r_0$,
\[
\fint_{\pa^+ B_r^+}\frac{\pa \Delta^{m-1} v^*}{\pa \nu}\,\ud S \le -  \frac{r^{-n}}{2} \int_{D_{1/2}\setminus D_r}|x|^{-\tau}{v^*(x,0)}^{p}\,\ud x.
\]
It follows that
\[
\fint_{\pa^+ B_{r_0}^+}\Delta^{m-1} v^*\,\ud S- \fint_{\pa^+ B_r^+}\Delta^{m-1} v^*\,\ud S\le -\frac{1}{2} \int_{r}^{r_0} \lda^{-n} \int_{D_{1/2}\setminus D _\lda}|x|^{-\tau}{v^*(x,0)}^{p}\,\ud x\ud \lda,
\]
which together with \eqref{eq:keylemma-a} yield
\be
\fint_{\pa^+ B_r^+}\Delta^{m-1} v^*\,\ud S \ge  r^{-n+1}
\ee
for all $0<r<r_1<r_0$, where $r_1$ is some fixed constant. For the same reason above, we have $L^{m-1} \bar {v^*}(r)\ge  r^{1-n}$.

For each case, from Lemma $\ref{lem:induction}$ we obtain
\be
\bar {v^*}(r)\ge c_2r^{2m-1-n} \quad\text{or}\quad\bar {v^*}(r) \le -c_2 r^{2m-1-n}
\ee
provided $r$ is sufficiently small. The later case can not happen because of the positivity of $v^*$.
The former case can not happen either because of Lemma \ref{lem:F}.

Therefore, we complete the proof of Proposition \ref{prop:integrability}.

\end{proof}
\section{Proof of main theorem}

\begin{prop}\label{prop:splitting} Under the assumptions in Proposition \ref{prop:integrability}, we have
\[
{u^*}(x,t)= \sum_{k=1}^{m-1}|X|^{2m-n-1}\left(\frac{t}{|X|^2}\right)^{2k}P_{2k}\left(\frac{x}{|X|^2}\right)+\gamma(n,m)
\int_{\R^n}\frac{|y|^{-\tau}{u^*}(y,0)^{p}}{(t^2+|x-y|^2)^{\frac{n-2m+1}{2}}}\,\ud y
\]
where $P_{2k}$ is a polynomial of $x$ with degree $\leq 2m-2-2k$.
\end{prop}

\begin{proof} Define
 \be\label{eq:V}
 V(x,t):=\gamma(n,m) \int_{\R^n}\frac{|y|^{-\tau}{u^*}(y,0)^{p}}{(t^2+|x-y|^2)^{\frac{n-2m+1}{2}}}\,\ud y.
 \ee
In view of \eqref{eq:goal} and $|y|^{-\tau}u^*(y,0)^p=O(|y|^{-(n+2m-1)})$ as $y\to \infty$, $V$ is well defined.  Set  $W:={u^*}- V$. By Lemma \ref{lem:Poisson-neu}, $W$ satisfies
\[
\begin{cases}
\Delta^{m} W=0 \quad &\mbox{in }\R^{n+1}_+,\\
\pa_t W=\pa_t \Delta W=\cdots=\pa_t\Delta^{m-1}W=0 \quad &\mbox{on }\partial\R^{n+1}_+\backslash\{0\}.
\end{cases}
\]
Since ${u^*}\geq 0$, then $W\geq -V$.  By Lemma \ref{lem:V} we obtain  $V$ is in weak$-L^{\frac{n+1}{n-2m+1}}(\R_+^{n+1})$. It follows from Corollary \ref{cor:Bocher1} that
\begin{align}\label{eq:W1}
W(X)=\sum_{|\al|\leq 2m-2}c_\al D^\al\Phi(X),
\end{align}
where $c_\al$ are constants and the $(n+1)$-th component of each $\al$ is even. By the definition of $\Phi(X)$ and $2m<n+1$, $D^\al\Phi$ can be rewritten as
\[D^\al\Phi(X)=\sum_{\beta\leq \alpha}c_\beta X^\beta|X|^{2m-n-1-2|\beta|}=\sum_{\beta\leq \alpha}c_\beta \left(\frac{X}{|X|^2}\right)^\beta|X|^{2m-n-1}.\]
where $\beta\leq \alpha$ means $\beta_i\leq \alpha_i$ for all $1\leq i\leq n$. Grouping and reordering the terms  according to the degree of $t$ in \eqref{eq:W1} yield
\begin{align}
W=\sum_{k=0}^{m-1}|X|^{2m-n-1}\left(\frac{t}{|X|^2}\right)^{2k}P_{2k}\left(\frac{x}{|X|^2}\right).
\end{align}
where $P_{2k}$ is  a polynomial on $x$ with degree $\leq 2m-2-2k$. Then the proposition follows from:

\textbf{Claim:} $P_0(x)\equiv 0$.

Let $l_0=\text{deg }P_0\geq 0$. Collect all the terms of degree $l_0$ in $P_0$ to be a homogeneous polynomial $\tilde{P}_0$.

If there is a nonempty open cone $\mathcal{S}\subset \R^n$ with $0$ as the vertex such that  $\tilde P_0(\frac{x}{|x|^2})> c>0$ on $\mathcal{S}\cap D_{r_0}$ for some constant $c$, then we can find $r_0>0$ small enough such that
\begin{align}
|P_0(x/|x|^2)-\tilde{P}_0(x/|x|^2)|<\frac12 \tilde{P}_0(x/|x|^2) \quad \mbox{in }   \mathcal{S}\cap D_{r_0}.
\end{align}
Therefore
\[{u^*}(x,0)= W(x,0)+V(x,0)\geq W(x,0)\geq \frac 12|x|^{2m-n-1}\tilde P_0\left(\frac{x}{|x|^2}\right)\quad \mbox{in } \mathcal{S} \cap  D_{r_0}.  \]
It leads to
\begin{align*}
\int_{D_1}|y|^{-\tau}{u^*}(y,0)^p\, \ud y\geq& \int_{\mathcal{S} \cap  D_{r_0} }|y|^{-\tau+p(2m-n-1)}\tilde P_0\left(\frac{y}{|y|^2}\right)^p\ud y\\
\geq& c^p \int_{\mathcal{S} \cap  D_{r_0} }|y|^{-n-2m+1}=\infty
\end{align*}
which contradicts to Proposition \ref{prop:integrability}. By the homogeneity of $\tilde P_0$, we conclude  $\tilde P_0(\frac{x}{|x|^2})\le 0$.

Suppose that $P_0(\frac{x}{|x|^2})\le 0$ but not identical to $0$.  Without loss of generality, one may assume $\inf_{|x|=1}\tilde P_0(x)=-1$ and denote cone $E:=\{x\in \R^n:\tilde{P}_0(\frac{x}{|x|^2})<-\frac 12|x|^{-l_0}\}$.
For the same fake, we can find $r_0>0$ small enough such that
\begin{align}
|P_0(x/|x|^2)-\tilde{P}_0(x/|x|^2)|<\frac12 \tilde{P}_0(x/|x|^2) \quad \mbox{in }   E\cap D_{r_0}.
\end{align}
Moreover, there exists $\varepsilon_0>0$ such that
\[\left|\{D_r\cap E\}\right|\geq \varepsilon_0 r^{n} \quad \forall ~0<r<1.\]
For some $\lambda>0$ to be chosen later, let $\rho=(4\lambda)^{-1/(n+l_0+1-2m)}$. On $D_\rho\cap E$, there holds
\[|x|^{2m-n-1}P_0\left(\frac{x}{|x|^2}\right)\leq \frac 12|x|^{2m-n-1}\tilde P_0\left(\frac{x}{|x|^2}\right)\leq -\frac 14 |x|^{2m-n-1-l_0}<-\lambda. \]
Therefore by noticing $W(x,0)=|x|^{2m-n-1}P_0(x/|x|^2)$, we have
\begin{align}\label{eq:54}
\left|\{x\in\R^n:W(x,0)<-\lambda\}\right|\geq \left|\{D_\rho\cap E\}\right|\geq \varepsilon_0(4\lambda)^{-\frac{n}{n-2m+1+l_0}}.
\end{align}

Decompose $V(x,0)$ as
\[V(x,0)=\int_{|y|\leq \delta}\frac{|y|^{-\tau}{u^*}(y,0)^p}{|x-y|^{n-2m+1}}\,\ud y+\int_{|y|>\delta}\frac{|y|^{-\tau}{u^*}(y,0)^p}{|x-y|^{n-2m+1}}\,\ud y:=V_1(x)+V_2(x),\]
where $\delta>0$ to be fixed.
For any $\va>0$, choose $\delta >0$ such that $\int_{D_\delta} |y|^{-\tau}{u^*}(y,0)^p<\va$.
From the weak type estimate of Riesz potential,
\be\label{eq:55} \left|\{x:V_1(x,0)>\frac{1}{2}\lambda\}\right|\leq C(m,n)(\varepsilon\lambda^{-1})^{\frac{n}{n-2m+1}},
\ee
where $C(m,n)>0$ depends only on $m$ and $n$.
Since $|y|^{-\tau}{u^*}(y,0)^p$ is smooth and bounded outside $D_\delta$,  $V_2$ is bounded. It follows that for $\lambda \ge 100\|V_2\|_{L^\infty}+1$,
\be \label{eq:est_V}
\left|\{x:V_2(x,0)>\frac 12\lambda\}\right|\le C(\varepsilon\lambda^{-1})^{\frac{n}{n-2m+1}},
\ee
where $C$ is independent of $\va$.  Combining \eqref{eq:55} and \eqref{eq:est_V}, we can choose $\va$ even small such that
\be \label{eq:57}
\left|\{x:V(x,0)>\frac 12\lambda\}\right| \le \va_0 10^{-\frac{n}{n-2m+1+l_0}} \lambda^{-\frac{n}{n-2m+1}}
\ee
for all $\lda> 100\|V_2\|_{L^\infty}+1$.
Note that
\[\{x: W(x,0)<-\lambda,|V(x,0)|<\lambda/2\}\subset\{x:{u^*}(x,0)<0\}=\emptyset.\]
It follows from \eqref{eq:54} and \eqref{eq:57} that for sufficiently large $\lda$,
\begin{align*}
0&=\left|\{x:W(x,0)<-\lambda,|V(x,0)|\leq\lambda/2\}\right|\\
&\geq \left|\{x:W(x,0)<-\lambda\}\right|-\left|\{x:|V(x,0)|>\lambda/2\}\right|\\
&\geq \varepsilon_0(4\lambda)^{-\frac{n}{n-2m+1+l_0}}-\va_0 10^{-\frac{n}{n-2m+1+l_0}} \lambda^{-\frac{n}{n-2m+1}}
\\&>0.
\end{align*}
We obtain a contradiction again.  Hence, $\tilde P_0(\frac{x}{|x|^2})=0$ and thus the claim is proved.

Therefore, we complete the proof of Proposition \ref{prop:splitting}.

\end{proof}

\begin{proof}[Proof of Theorem \ref{thm:main-a}] Let $V$ be defined in \eqref{eq:V}. By Proposition \ref{prop:splitting}, $V(x,0)=u^*(x,0)$ and $V^*(x,0):=V_{0,1}(x,0)$ is smooth in $\R^n$. It follows from \eqref{eq:V} that
\[
V(x,t)= \gamma(n,m) \int_{\R^n}\frac{|y|^{-\tau}V(y,0)^{p}}{(t^2+|x-y|^2)^{\frac{n-2m+1}{2}}}\,\ud y,
\]
from lemma \ref{lem:neuman_conv}, it is equivalent to
\be \label{eq:V*}
V^*(x,t)= \gamma(n,m) \int_{\R^n}\frac{V^*(y,0)^{p}}{(t^2+|x-y|^2)^{\frac{n-2m+1}{2}}}\,\ud y.
\ee
on the condition that the right hand side integral converges. This is justified through
\begin{align*}
\int_{\R^n\backslash D_1}\frac{V^*(y,0)^{p}}{(t^2+|x-y|^2)^{\frac{n-2m+1}{2}}}\,\ud y&\leq C\int_{\R^n\backslash D_1}u(y,0)^p|y|^{-(n-2m+1)}\,\ud y\\
&=C\int_{D_1}|x|^{-\tau}u^*(x,0)^p\,\ud x<\infty.
\end{align*}

Sending $t\to 0$ in \eqref{eq:V*}, we see that
\[
V^*(x,0)= \gamma(n,m) \int_{\R^n}\frac{V^*(y,0)^{p}}{|x-y|^{n-2m+1}}\,\ud y.
\]
Since $V^*(x,0)$ is smooth in $\R^n$, it follows from Chen-Li-Ou \cite{CLO} and Li \cite{Li04} that
\begin{align*} V^*(x,0)=0 \quad \mbox{if }p<\frac{n+2m-1}{n-2m+1},
\end{align*} and
\[
V^*(x,0)=c_0(n,m)\left(\frac{\lda}{1+\lda^2|x-x_0|^2}\right)^\frac{n-2m+1}{2} \quad \mbox{for some }\lda\ge 0, x_0\in \R^n,
\]
where $c_0(n,m)>0$ is a constant depending only on $n,m$, if $p=\frac{n+2m-1}{n-2m+1}$.  One may also apply the moving planes or spheres method to \eqref{eq:V*} directly to prove the classification result; see Dou-Zhu \cite{DZ}.  By Proposition \ref{prop:splitting}, Theorem \ref{thm:main-a} follows immediately.

\end{proof}

\section{An application to conformal geometry}
\label{sec:appl}

Given Theorem \ref{thm:main-a}, we construct metrics which is singular on single boundary point of the unit ball below. Define the map  $F:\R^{n+1}_+\to B_1$ by
\[
F(x,t)= \left(\frac{2 x}{|x|^2+(t+1)^2}, \frac{|X|^2-1}{|x|^2+(t+1)^2}\right).
\]
Observe that $F(x,0)\to \mathbb{S}^n$,
\[
F(x,0)= \left(\frac{2 x}{|x|^2+1},\frac{|x|^2-1}{|x|^2+1}\right)
\]
is the inverse of the stereographic projection.
Let
\[
v(F(X))= |J_F|^{-\frac{n-2m+1}{2}} u(X)
\]
wherer $|J_F|$ is the Jacobian determinant of $F$.
\begin{prop}\label{prop:connection} Assume the assumptions in Theorem \ref{thm:main-a}. Suppose that $u>0$ in $\R^{n+1}_+\cup \pa \R^{n+1}_+$ and $p=\frac{n+2m-1}{n-(2m-1)}$. Let $v$ be defined as above and $g=v^{\frac{4}{n-(2m-1)}}\ud X^2$ in $B_1$ be a conformal metric of the flat metric. Then the $2m$-th order $Q$-curvature of $g$ in $B_1$ is zero and the boundary $(2m-1)$-th order $Q$-curvature is constant on $\pa B_1\setminus \{(0,1)\}$.

If the polynomial part of in the conclusion of the Theorem \ref{thm:main-a} is nontrivial,  then $v$ blows up near the boundary point $(0,1)$.
\end{prop}
\begin{proof} By the conformal invariance, it is easy to check that $\Delta^m v=0$, see Li-Mastrolia-Monticelli \cite{LMM}. It follows that the $2m$-th order $Q$-curvature of $g$ in $B_1$ is zero.

By the proof of Proposition \ref{prop:splitting} we see that
\[
u(x,0)=\gamma(n,m) \int_{\R^n}\frac{u(y,0)^{\frac{n+2m-1}{n-(2m-1)}}}{|x-y|^{n-(2m-1)}}\,\ud y.
\]
It follows that
\[
v(X)= \gamma(n,m) \int_{\pa B_1} \frac{v(Y,0)^{\frac{n+2m-1}{n-(2m-1)}}}{|X-Y|^{n-(2m-1)}}\,\ud S_Y
\]
and thus the $(2m-1)$-th order $Q$-curvature is a constant, see Jin-Li-Xiong \cite{JLX3} for more details.

If the polynomial part of in the conclusion of the Theorem \ref{thm:main-a} is nontrivial, by the definition of $v$ it is easy to see $v$  blows up near the boundary point $(0,1)$.
\end{proof}

\begin{rem} Note that the scalar curvature metric $g$ in Proposition \ref{prop:connection}  could be negative.
\end{rem}

If $m=2$, we have explicit equations of $v$, see Chang-Qing \cite{CQ}, Branson-Gover \cite{BG} and Case \cite{Case}:
\begin{align} \label{eq:ball}
\begin{cases}
\Delta^2 v=0 &\quad \mbox{in }B_1(0,1),\\
\textbf{B}^3_1v=0&\quad \mbox{on }\pa B_1(0,1)\setminus \{(0,1)\},\\
\textbf{B}^3_3 v= v^{\frac{n+3}{n-3}} &\quad \mbox{on }\pa B_1(0,1)\setminus \{(0,1)\},
\end{cases}
\end{align}
where
\[
\textbf{B}^3_1v=\frac{\pa v}{\pa\nu}+\frac{n-3}{2} v,
\]
\[
\textbf{B}_3^3 v=-\frac{\pa\Delta v}{\pa\nu}-\frac{n-3}{2}\frac{\pa^2 v}{\pa\nu^2}-\frac{3n-5}{2}\Delta_{\mathbb{S}^n} v+\frac{3n^2-7n+6}{4}\frac{\pa v}{\pa\nu}+\frac{n^2-n+2}{4}\frac{n-3}{2}v.
\]
Therefore, the metric $g$ has flat $4$-th order $Q$-curvature, flat mean curvature and constant $3$-th order $Q$-curvature on the boundary.
By Theorem \ref{thm:main-a}, solutions of \eqref{eq:ball} satisfying
\[
v(F(X))=o(|X|^n)
\]
are classified.
If $m\ge 3$, the analogues of \eqref{eq:ball} can be found in  Branson-Gover \cite{BG} but are more complicated. Similarly,  Theorem \ref{thm:main-a} can be applied to them.

\bigskip

\noindent L. Sun

\noindent Department of Mathematics, Rutgers University\\
110 Frelinghuysen Road, Piscataway, NJ 08854, USA\\[1mm]
Email: \textsf{ls680@math.rutgers.edu}

\medskip

\noindent J. Xiong

\noindent School of Mathematical Sciences, Beijing Normal University\\
Beijing 100875, China\\[1mm]
Email: \textsf{jx@bnu.edu.cn}

\end{document}